\documentclass[a4paper, 11pt]{article}
\usepackage[T1]{fontenc}
\usepackage[utf8]{inputenc}
\usepackage{amsmath}
\usepackage{amsfonts}
\usepackage{amssymb}
\usepackage{amsthm}
\usepackage{graphicx}
\usepackage[english]{babel}
\usepackage{version}
\usepackage{nicefrac}
\usepackage{tipa} 
\usepackage{hyperref}

\theoremstyle{definition}
\newtheorem{defin}{Definition}[section]
\newtheorem{ex}[defin]{Example}
\theoremstyle{plain}
\newtheorem{theo}[defin]{Theorem}
\newtheorem{lemma}[defin]{Lemma}
\newtheorem{obs}[defin]{Remark}
\newtheorem{prop}[defin]{Proposition}
\newtheorem{cor}[defin]{Corollary}
\newtheorem{theorem}{Theorem}

\numberwithin{equation}{section}

\renewenvironment{abstract}
{\par\noindent\textbf{\abstractname.}\ \ignorespaces}
{\par\medskip}

\title{Entropies of non-positively curved metric spaces}
\author{Nicola Cavallucci}
\date{}

\begin{document}
	\maketitle
	\begin{abstract}
		\footnotesize
		We show the equivalences of several notions of entropy, like a version of the topological entropy of the geodesic flow and the Minkowski dimension of the boundary, in metric spaces with convex geodesic bicombings satisfying a uniform packing condition. Similar estimates will be given in case of closed subsets of the boundary of Gromov-hyperbolic metric spaces with convex geodesic bicombings.	
	\end{abstract}

\tableofcontents

\section{Introduction}
\label{sec-introduction}
This paper is devoted to the investigation of different asymptotic quantities associated to a metric space, some of them classical and widely studied. As we will see in a minute these invariants have different nature: dynamical, measure-theoretic and combinatoric.
The purpose is two-fold:
\begin{itemize}
	\item[-] to show the relations between these invariants,  especially to understand when they are equal;
	\item[-] to develop interesting tools and techniques to simplify the computation of these invariants. 
\end{itemize}
The second goal is extremely useful for applications: it is important to have flexible and easy to compute definitions in order to study these invariants in concrete cases. In the forthcoming paper \cite{Cav21bis} we will use this flexibility to extend Otal-Peigné's Theorem (\cite{OP04}) to a large class of metric spaces.
A simplified version of the techniques developed on this paper were used to show the continuity of the critical exponent of fundamental groups of a large class of compact metric spaces under Gromov-Hausdorff convergence, see \cite{Cav21ter}.
\vspace{2mm}

\noindent We are mainly interested on metric spaces satisfying weak bounds on the curvature. As upper bound we consider a very weak convexity condition: the existence of a convex geodesic bicombing. A geodesic bicombing on a metric space $X$ is a map $\sigma \colon X\times X \times [0,1] \to X$ such that for all $x,y\in X$ the function $\sigma_{xy}(\cdot)=\sigma(x,y,\cdot)$ is a geodesic (parametrized proportionally to arc-length) from $x$ to $y$. A bicombing $\sigma$ is convex if for all $x,y,x',y'\in X$ the map $t\mapsto d(\sigma_{xy}(t), \sigma_{x'y'}(t))$ is convex. 
Among metric spaces admitting a convex bicombing there are CAT$(0)$-spaces, Busemann convex spaces and all normed vector spaces. The interest in this condition is given by its stability under limits (\cite{Des15}, \cite{CavS20bis}), while it is not the case for Busemann convex spaces.
Given a bicombing $\sigma$, every curve $\sigma_{xy}$ is called a $\sigma$-geodesic. The bicombing is geodesically complete if any $\sigma$-geodesic can be extended to a bigger $\sigma$-geodesic. This notion coincides with the usual geodesic completeness in case of Busemann convex metric spaces. A GCB-space is a couple $(X,\sigma)$ where $X$ is a complete metric space and $\sigma$ is a geodesically complete, convex, geodesic bicombing on $X$.\\
As a lower bound on the curvature we take a uniform packing condition: a metric space $X$ is said to be $P_0$-packed at scale $r_0$ if for all $x\in X$ the cardinality of a maximal $2r_0$-separated subset of $B(x,3r_0)$ is at most $P_0$. This uniform packing condition interacts very well with the weak convexity property given by a geodesically complete, convex, geodesic bicombing, implying a uniform control of the packing condition \emph{at every scale} (see Proposition \ref{packingsmallscales} and \cite{CavS20bis}).
Sometimes, especially in the second part of the paper where relative versions of the invariants will be studied, we will impose also a Gromov-hyperbolicity condition on our metric space.

\subsection{Lipschitz-topological entropy of the geodesic flow}
The topological entropy of the geodesic flow has been intensively studied in case of Riemannian manifolds, especially in the negatively curved setting. If such a manifold is denoted by $\bar{M} = M/\Gamma$, where $M$ is its universal cover and $\Gamma$ is its fundamental group, then the set of parametrized geodesic lines is identified with the unit tangent bundle $S\bar{M}$.
Probably the most important invariant associated to the geodesic flow is its topological entropy, denoted $h_\text{top}(\bar{M})$. It equals the Hausdorff dimension of the limit set of $\Gamma$ and the critical exponent of $\Gamma$ (see \cite{Sul84}, \cite{OP04}). Moreover if $\bar{M}$ is compact then it coincides also with the volume entropy of $M$ (\cite{Man79}), while this is no more true in general, even when $\bar M$ has finite volume (cp. \cite{DPPS09}): we will come back to these examples at the end of the introduction. 
\vspace{2mm}

\noindent In case of GCB-metric spaces $(X,\sigma)$ we restrict the attention to $\sigma$-geodesic lines. The topological entropy of the $\sigma$-geodesic flow is defined as the topological entropy (in the sense of Bowen, cp. \cite{Bow73}, \cite{HKR95}) of the dynamical system $(\text{Geod}_\sigma(X), \Phi_t)$, where $\text{Geod}_\sigma(X)$ is the space of parametrized $\sigma$-geodesic lines, endowed with the topology of uniform convergence on compact subsets, and $\Phi_t$ is the reparametrization flow. It is:
$$h_{\text{top}}(\text{Geod}_\sigma(X)) = \inf_{\text{\texthtd}}\sup_{K\subseteq \text{Geod}_\sigma(X)} \lim_{r\to 0} \lim_{T\to +\infty} \frac{1}{T}\log \text{Cov}_{\text{\texthtd}^T}(K,r),$$
where the infimum is taken {\em among all metrics on \textup{Geod}$_\sigma(X)$} inducing its topology, the supremum is taken among all compact subsets of Geod$_\sigma(X)$, \texthtd$^T$ is the distance $\text{\texthtd}^T(\gamma,\gamma') = \max_{t\in [0,T]}\text{\texthtd}(\Phi_t(\gamma), \Phi_t(\gamma'))$ and $\text{Cov}_{\text{\texthtd}^T}(K,r)$ is the minimal number of balls (with respect to the metric \texthtd$^T$) of radius $r$ needed to cover $K$. 
We remark that in case of Busemann convex (or CAT$(0)$) metric spaces the space of $\sigma$-geodesic lines coincides with the set of geodesic lines. This flow has no recurrent geodesics, so applying the variational principle (cp. \cite{HKR95}) it is straightforward to conclude that its topological entropy is zero (Lemma \ref{zero-entropy}). Looking carefully at the proof of the variational principle it turns out that the metrics on $\text{Geod}_\sigma(X)$ almost realizing the infimum in the definition of the topological entropy are restriction to Geod$_\sigma(X)$ of metrics on its one-point compactification. In particular they are no the natural ones to consider in this setting.
That is why, in Section \ref{sec-liptop}, we will restrict the attention to the class of {\em{geometric metrics}} \texthtd: those with the property that the evaluation map $E\colon (\text{Geod}_\sigma(X), \text{\texthtd}) \to (X,d)$ defined as $E(\gamma)=\gamma(0)$ is Lipschitz. Notice that for a geometric metric two geodesic lines are not close if they are distant at time $0$. Accordingly the \emph{Lipschitz-topological entropy} of the geodesic flow is defined as
$$h_{\text{Lip-top}}(\text{Geod}_\sigma(X)) = \inf_{\text{\texthtd}}\sup_{K\subseteq \text{Geod}_\sigma(X)} \lim_{r\to 0} \lim_{T\to +\infty} \frac{1}{T}\log \text{Cov}_{\text{\texthtd}^T}(K,r),$$
where \emph{the infimum is taken only among the geometric metrics of} Geod$_\sigma(X)$.
Although the definition of the Lipschitz-topological entropy is quite complicated, its computation can be remarkably simplified. Indeed one of the most used metric on Geod$_\sigma(X)$ (see for instance \cite{BL12}) is:
$$d_{\text{Geod}}(\gamma,\gamma') = \int_{-\infty}^{+\infty}d(\gamma(s),\gamma'(s))\frac{1}{2e^{\vert s \vert}} ds$$
that induces the topology of $\text{Geod}_\sigma(X)$ and is geometric, and it turns out that it realizes the infimum in the definition of the Lipschitz-topological entropy.

\begin{theorem}[Extract from Theorem \ref{lip-top_cov} \& Proposition \ref{entropyrelations}]
	\label{A}
	Let $(X,\sigma)$ be a \textup{GCB}-space that is $P_0$-packed at scale $r_0$.
	Then
	$$h_{\textup{Lip-top}}(\textup{Geod}_\sigma(X)) = \lim_{T \to +\infty}\frac{1}{T}\log \textup{Cov}_{d_{\textup{Geod}}^T}(\textup{Geod}_\sigma(x),r_0),$$
	where $\textup{Geod}_\sigma(x)$ is the set of $\sigma$-geodesic lines passing through $x$ at time $0$.	
\end{theorem}
\noindent Therefore the infimum in the definition of the Lipschitz topological entropy is actually realized by the metric $d_{\text{Geod}}$ and the supremum among the compact sets can be replaced by a fixed (relatively small) compact set. Moreover also the scale $r$ can be fixed to be $r_0$ (or any other positive real number).\\
Actually the result of Theorem \ref{A} is still valid for a whole family of metrics on Geod$_\sigma(X)$: this flexibility will be one of the fundamental ingredient in the main result of \cite{Cav21bis}.

\subsection{Volume and Covering entropy}
The second definition of entropy we consider (see Section \ref{subsec-volume}) is the volume entropy.
If $X$ is a metric space equipped with a measure $\mu$ it is classical to consider the exponential growth rate of the volume of balls, namely:
$$h_\mu(X) :=\lim_{T\to+\infty}\frac{1}{T}\log \mu(\overline{B}(x,T)).$$
It is called the {\em volume entropy} of $X$ with respect to the measure $\mu$ and it does not depend on the choice of the basepoint $x\in X$ by triangular inequality. 
This invariant has been studied intensively in case of complete Riemannian manifolds with non positive sectional curvature, where $\mu$ is the Riemannian volume on the universal cover.
It is related to other interesting invariants as the simplicial volume of the manifold (see \cite{Gro82}), \cite{BS20}), a macroscopical condition on the scalar curvature (cp. \cite{Sab17}) and the systole in case of compact, non-geometric $3$-manifolds (cp. \cite{CerS19}).
Moreover the infimum of the volume entropy among all the possible Riemannian metrics of volume $1$ on a fixed closed manifold is a subtle homotopic invariant (see \cite{Bab93}, \cite{Bru08} for general considerations and \cite{BCG95}, \cite{Pie19} for the computation of the minimal volume entropy in case of, respectively, closed $n$-dimensional manifolds supporting a locally symmetric metric of negative curvature and $3$-manifolds). 
\vspace{2mm}

\noindent A measure $\mu$ is called $H$-homogeneous at scale $r$ if
$$\frac{1}{H} \leq \mu(\overline{B}(x,r)) \leq H$$
for every $x\in X$. Among homogeneous measures there is a remarkable example: the volume measure $\mu_X$ of a complete, geodesically complete, CAT$(0)$ metric space $X$ that is $P_0$-packed at scale $r_0$ (see \cite{CavS20} and \cite{LN19} for a description of the measure). In case $X$ is a Riemannian manifold of non-positive sectional curvature then $\mu_X$ coincides with the Riemannian volume, up to a universal multiplicative constant. \\
A more combinatoric and intrinsic version of the volume entropy of a generic metric space is the \emph{covering entropy}, defined as:
$$h_{\text{Cov}}(X):=\lim_{T\to+\infty}\frac{1}{T}\log \text{Cov}(\overline{B}(x,T), r),$$
where $x$ is a point of $X$ and $\text{Cov}(\overline{B}(x,T), r)$ is the minimal number of balls of radius $r$ needed to cover $\overline{B}(x,T)$. It does not depend on $x$ but it can depend on the choice of $r$. This is not the case when $X$ is a GCB-space that is $P_0$-packed at scale $r_0$, as follows by Proposition \ref{cov-ent-pack}. Moreover it is always finite (cp.Lemma \ref{entropy-bound}).

\subsection{Minkowski dimension of the boundary}
The expression of the Lipschitz-topological entropy given by Theorem \ref{A} suggests the possibility to relate that invariant to some property of the boundary at infinity of $X$.
For simplicity we suppose $X$ is also Gromov-hyperbolic, so that the boundary at infinity is metrizable.
If we denote by $(\cdot,\cdot)_x$ the Gromov product based on $x$ then the generalized visual ball of center $z\in \partial X$ and radius $\rho$ is $B(z,\rho)=\lbrace z'\in \partial X \text{ s.t. } (z,z')_x> \log\frac{1}{\rho}\rbrace$. The {\em visual Minkowski dimension} of the Gromov boundary $\partial X$ is:
$$\text{MD}(\partial X) = \lim_{T\to +\infty}\frac{1}{T}\log\text{Cov}(\partial X, e^{-T}),$$
where $\text{Cov}(\partial X, e^{-T})$ is the minimal number of generalized visual balls of radius $e^{-T}$ needed to cover $\partial X$. If the generalized visual balls are metric balls for some visual metric $D_{x,a}$ then we refind the usual definition of Minkowski dimension of the metric space $(\partial X, D_{x,a})$, once the obvious change of variable $\rho = e^{-T}$ is made. This invariant is presented in Section \ref{subsec-minkowski}. 

%
%

\subsection{Equality of the entropies}
\label{subsec-intro-asymp}
One of our main results is:
\begin{theorem}
	\label{B}
	Let $(X,\sigma)$ be a \textup{GCB}-space that is $P_0$-packed at scale $r_0$. Then
	$$h_{\textup{Lip-top}}(\textup{Geod}_\sigma(X)) = h_\mu(X) = h_\textup{Cov}(X),$$
	where $\mu$ is every homogeneous measure on $X$.
	Moreover if $X$ is also $\delta$-hyperbolic then the quantities above coincide also with $\textup{MD}(\partial X)$.
\end{theorem}
Actually something more is true but in order to state it we need to recall the notion of equivalent asymptotic behaviour of two functions introduced in \cite{Cav21ter}. Given $f,g\colon [0,+\infty) \to \mathbb{R}$ we say that {\em $f$ and $g$ have the same asymptotic behaviour}, and we write $f\asymp g$, if for all $\varepsilon > 0$ there exists $T_\varepsilon \geq 0$ such that if $T\geq T_\varepsilon$ then $\vert f(T) -  g(T) \vert \leq \varepsilon$. The function $T_\varepsilon$ is called the \emph{threshold function}.
Usually we will write $f\underset{P_0,r_0,\delta,\ldots}{\asymp} g$ meaning that the threshold function can be expressed only in terms of $\varepsilon$ and $P_0,r_0,\delta,\ldots$
\begin{theorem}
	\label{C}
	Let $(X,\sigma)$ be a \textup{GCB}-space that is $P_0$-packed at scale $r_0$. Then the functions defining the quantities of Theorem \ref{B} have the same asymptotic behaviour and the threshold functions depend only on $P_0,r_0,\delta$ and the homogeneous constants of $\mu$.
\end{theorem}
Therefore not only all the introduced quantities define the same number, but all of them also have the same asymptotic behaviour. This means that if one can control the rate of convergence to the limit of one of these quantities then also the rate of convergence of all the other quantities is bounded.
We remark that, differently from many of the papers in the literature, we do not require any group action on our metric spaces. \\
The control of the asymptotic behaviour of the function defining the Minkowski dimension is the main ingredient of the continuity theorem proved in \cite{Cav21ter}. The same ideas can be used to show similar continuity statements in more general settings, but we won't explore these applications here.

\subsection{Entropies of the closed subsets of the boundary}
In case $X$ is $\delta$-hyperbolic it is possible to define the versions of all the different notions of entropies relative to subsets of the boundary $\partial X$. 
For every subset $C\subseteq \partial X$ we denote by Geod$_\sigma(C)$ the set of parametrized $\sigma$-geodesic lines with endpoints belonging to $C$ and with QC-Hull$(C)$ the union of the points belonging to the geodesics joining any two points of $C$. Actually the hyperbolicity assumption (or at least a visibility assumption on $\partial X$) is necessary since otherwise the sets Geod$_\sigma(C)$ and QC-Hull$(C)$ could be empty.
The numbers
\begin{equation*}
	\begin{aligned}
		h_{\text{Cov}}(C) &= \lim_{T\to +\infty}\frac{1}{T}\log \text{Cov}(\overline{B}(x,T) \cap \text{QC-Hull}(C), r_0) \\
		h_{\text{Lip-top}}(\text{Geod}_\sigma(C)) &= \inf_{\text{\texthtd}}\sup_{K\subseteq \text{Geod}_\sigma(C)} \lim_{r\to +\infty} \lim_{T\to +\infty} \frac{1}{T}\log \text{Cov}_{\text{\texthtd}^T}(K,r) \\
		\text{MD}(C) &= \lim_{T\to +\infty}\frac{1}{T}\log\text{Cov}(C, e^{-T})
	\end{aligned}
\end{equation*}
are called, respectively, covering entropy of $C$, Lipschitz-topological entropy of Geod$_\sigma(C)$ and visual Minkowski dimension of $C$. The volume entropy of $C$ with respect to a measure $\mu$ is
$$h_\mu(C) = \sup_{\tau \geq 0}\lim_{T\to +\infty}\frac{1}{T}\log \mu(\overline{B}(x,T) \cap  \overline{B}(\text{QC-Hull}(C), \tau)),$$
where $\overline{B}(Y,\tau)$ is the closed $\tau$-neighbourhood of $Y\subseteq X$. If $\mu$ is $H$-homogeneous at scale $r$ then the volume entropy can be computed taking $\tau = r$ in place of the supremum (Proposition \ref{closed-volume-entropy}). For instance when $X$ is a Riemannian manifold with pinched negative curvature then the the Riemannian volume $\mu_X$ is $H(r)$-homogeneous at every scale $r>0$, so the definition does not depend on $\tau$ at all.  Most of the relations of Theorem \ref{B} remain true for subsets of the boundary, but the asymptotic behaviour of the different functions involved in the definitions of the entropies depend also on the choice of the basepoint $x\in X$. The best possible choice, $x\in \text{QC-Hull}(C)$, allows us to give again uniform asymptotic estimates.
\begin{theorem}
	\label{E}
	Let $(X,\sigma)$ be a $\delta$-hyperbolic \textup{GCB}-space that is $P_0$-packed at scale $r_0$ and let $C\subseteq \partial X$. Then
	$$h_\textup{Cov}(C) = \textup{MD}(C) = h_\mu(C)$$
	for every homogeneous measure $\mu$ on $X$.
	All the functions defining the quantities above have the same asymptotic behaviour and the threshold functions can be expressed only in terms of $P_0,r_0,\delta$ and the homogeneous constants of $\mu$, if the basepoint $x$ belongs to \textup{QC-Hull}$(C)$.
\end{theorem}
The proof of this result does not follow by the same arguments of Theorem \ref{B}, indeed it will be based heavily on the Gromov-hyperbolicity of $X$.
The relation between the Lipschitz-topological entropy of Geod$_\sigma(C)$ and the other definitions of entropy is more complicated. We have
\begin{theorem}
	\label{G}
	Let $(X,\sigma)$ be a $\delta$-hyperbolic \textup{GCB}-space that is $P_0$-packed at scale $r_0$ and let $C\subseteq \partial X$. Then
	\begin{itemize}
		\item[(i)] if $C$ is closed then $h_{\textup{Lip-top}}(\textup{Geod}_\sigma(C)) = h_\textup{Cov}(C)$ and the functions defining these two quantities have the same asymptotic behaviour with thresholds function depending only on $P_0,r_0,\delta$.\\
		More precisely, if $x\in \textup{QC-Hull}(C)$ then $$h_{\textup{Lip-top}}(\textup{Geod}_\sigma(C)) = \lim_{T \to +\infty}\frac{1}{T}\log \textup{Cov}_{d_{\textup{Geod}}^T}(\textup{Geod}_\sigma(\overline{B}(x,22\delta), C),r_0),$$
		where $\textup{Geod}_\sigma(\overline{B}(x,22\delta), C)$ is the set of geodesic lines with endpoints in $C$ and passing through $\overline{B}(x,22\delta)$ at time $0$.
		\item[(ii)] if $C$ is not closed then 
		$$h_{\textup{Lip-top}}(\textup{Geod}_\sigma(C)) = \sup_{C'\subseteq C} h_{\textup{Lip-top}}(\textup{Geod}_\sigma(C')) \leq h_{\textup{Cov}}(C),$$
		where the supremum is taken among the closed subsets of $C$.
	\end{itemize}
\end{theorem}
The inequality in (ii) can be strict, as shown in \cite{Cav21bis}.

\subsection{Differences of the invariants for geometrically finite groups}
\label{subsec-intro}
In this last part of the introduction we restrict the attention to the case of Riemannian manifolds $\bar{M}=M/\Gamma$ with pinched negative sectional curvature. If $\Gamma$ is geometrically finite then the limit set of $\Gamma$ is the union of the radial limit set $\Lambda_\text{r}(\Gamma)$ and the bounded parabolic points. The latter is a countable set, therefore the Hausdorff dimension of the limit set coincides with the Hausdorff dimension of the radial limit set. So by Bishop-Jones' Theorem it holds (here HD denotes the Hausdorff dimension):
\begin{equation}
	\label{HD-geom-finit}
	\text{HD}(\Lambda(\Gamma)) = \text{HD}(\Lambda_r(\Gamma)) = h_\Gamma.
\end{equation}
We remark that this is not true if $\Gamma$ is not geometrically finite, even when $M$ is the hyperbolic space.
\begin{ex}
	In general it can happen 
	$\text{HD}(\Lambda_\text{r}(\Gamma))) <\text{HD}(\Lambda(\Gamma))$. Indeed let $\Gamma$ be a cocompact group of $\mathbb{H}^2$ and let $\Gamma'$ be a normal subgroup of $\Gamma$ such that $\Gamma/\Gamma'$ is not amenable. Let $F\subseteq \Lambda(\Gamma')$ be the subsets of points $z$ that are fixed by some $g \in \Gamma'$. For every $z\in F$ and every $h\in \Gamma$ we have $hz = hgz = g'hz$ for some $g'\in \Gamma'$ since $\Gamma'$ is normal. Then $hz$ is fixed by $g'$ and so it belongs to $F$, i.e. $F$ is $\Gamma$-invariant. By minimality of $\Lambda(\Gamma)$ we get $\Lambda(\Gamma') = \Lambda(\Gamma)$, so $\text{HD}(\Lambda(\Gamma')) = \text{HD}(\Lambda(\Gamma))$. But by the growth tightness of $\Gamma$ (cp. for instance \cite{Sam02}) we have \vspace{-1mm}
	$$\text{HD}(\Lambda_{\text{r}}(\Gamma')) = h_{\Gamma'} < h_\Gamma = \text{HD}(\Lambda(\Gamma)) = \text{HD}(\Lambda(\Gamma')).$$
\end{ex}
\noindent However if $M$ is the hyperbolic space and $\Gamma$ is geometrically finite then even something more is true, indeed by \cite{SU96}:
\begin{equation}
	\label{HD-MD-hyp}
	h_\Gamma=\text{HD}(\Lambda(\Gamma)) = \text{MD}(\Lambda(\Gamma)).
\end{equation}
This equality fails to be true for geometrically finite (actually of finite covolume) groups of manifolds with pinched, but variable, negative curvature. Indeed we have:
\begin{ex}
	\label{DPPS-example}
	In \cite{DPPS09} it is presented an example of a smooth Riemannian manifold $M$ with pinched negative sectional curvature admitting a (non-uniform) lattice (i.e. a group of isometries $\Gamma$ with $\text{Vol}(M/\Gamma) < + \infty)$ such that $h_\Gamma < h_{\mu_M}(M)$ (recall that $\mu_M$ denotes the Riemannian volume of $M$). We observe that since $\Gamma$ is a lattice then $\Lambda(\Gamma)=\partial M$, so $h_{\mu_M}(M) = \text{MD}(\Lambda(\Gamma))$ by Theorem \ref{B}, while $h_\Gamma=\text{HD}(\Lambda_\text{r}(\Gamma)) = \text{HD}(\Lambda(\Gamma))$ by \eqref{HD-geom-finit}.
\end{ex}
\noindent The example above is due to a relevant variation of the curvature of $M$. Indeed in \cite{DPPS19} is shown that for non-uniform lattices $\Gamma$ of asymptotically $1/4$-pinched manifolds with negative curvature $M$ it holds $h_{\mu_M}(M)=h_\Gamma$.
\noindent The general situation in the geometrically finite case is:
\begin{equation}
	\label{aaaaaaa}
	\begin{aligned}
		&\text{HD}(\Lambda(\Gamma)) = h_\Gamma = h_\text{top}(M/\Gamma) \\  h_\text{Lip-top}(\text{Geod}(\Lambda(\Gamma&))) = h_\text{Cov}(\Lambda(\Gamma)) = h_{\mu_M}(\Lambda(\Gamma)) = \text{MD}(\Lambda(\Gamma)),
	\end{aligned}
\end{equation}
where the equalities in the second line follow by Theorem \ref{E} and Theorem \ref{G}, while the equalities in the first line are consequences of \eqref{HD-geom-finit} and Otal-Peigne's variational principle \cite{OP04}.
Moreover it is clear that the first line is always less than or equal to the second one, since the Hausdorff dimension is always smaller than or equal to the Minkowski dimension. 
\vspace{2mm}

\noindent The relations in \eqref{aaaaaaa} allow us to give new interpretations of the phenomena occurring in Example \ref{DPPS-example}, i.e. the possible difference between the critical exponent of the group and the volume entropy of $\Lambda(\Gamma)$:
\begin{itemize}
	\item \emph{measure-theoretic interpretation:} it can be seen as the difference between the Hausdorff and the Minkowski dimension of the limit set $\Lambda(\Gamma)$, so it is related to the fractal structure of the limit set;
	\item \emph{dynamical interpretation:} it can be seen as the difference between the topological entropy of the geodesic flow of the quotient and the Lipschitz-topological entropy of Geod$(\Lambda(\Gamma))$. 
	\item \emph{combinatoric interpretation:} it can be seen as the difference between $h_\Gamma$ and $h_{\text{Cov}}(\Lambda(\Gamma))$, where the former counts the exponential growth rate of an orbit while the latter counts the exponential growth rate of the cardinality of $r$-nets, for some (any) $r>0$. Here the difference arises in terms of sparsity of the orbit.
\end{itemize}

\section{GCB-spaces and space of geodesics}
Throughout the paper $X$ will denote a metric space with metric $d$. The open (resp.closed) ball of radius $r$ and center $x$ is denoted by $B(x,r)$ (resp. $\overline{B}(x,r)$), while the metric sphere of center $x$ and radius $R$ is denoted by $S(x,R)$. We use the notation $A(x,r,r')$ to denote the closed annulus of center $x$ and radii $0<r<r'$, i.e. the set of points $y\in X$ such that $r\leq d(x,y) \leq r'$.
A geodesic segment is an isometry $\gamma\colon I \to X$ where $I=[a,b]$ is a a bounded interval of $\mathbb{R}$. The points $\gamma(a), \gamma(b)$ are called the endpoints of $\gamma$. A metric space $X$ is said geodesic if for all couple of points $x,y\in X$ there exists a geodesic segment whose endpoints are $x$ and $y$. We will denote any geodesic segment between two points $x$ and $y$, with an abuse of notation, as $[x,y]$. A geodesic ray is an isometry $\gamma\colon[0,+\infty)\to X$ while a geodesic line is an isometry $\gamma\colon \mathbb{R}\to X$. 
\vspace{2mm}

\noindent Let $Y$ be any subset of a metric space $X$:\\
-- a subset $S$ of $Y$ is called  {\em $r$-dense}   if   $\forall y \in Y$  $\exists z\in S$ such that $d(y,z)\leq r$; \\
-- a subset $S$ of $Y$ is called  {\em $r$-separated} if   $\forall y,z \in S$ it holds $d(y,z)> r$.\\
The packing number of $Y$ at scale $r$ is the maximal cardinality of a $2r$-separated subset of $Y$ and it is denoted by $\text{Pack}(Y,r)$. The covering number of $Y$ is the minimal cardinality of a $r$-dense subset of $Y$ and it is denoted by $\text{Cov}(Y,r)$. The following inequalities are classical:
\begin{equation}
	\label{packing-covering}
	\text{Pack}(Y,2r) \leq \text{Cov}(Y,2r) \leq \text{Pack}(Y,r).
\end{equation}
The packing and the covering functions of $X$ are respectively
$$\text{Pack}(R,r)=\sup_{x\in X}\text{Pack}(\overline{B}(x,R),r), \qquad \text{Cov}(R,r)=\sup_{x\in X}\text{Cov}(\overline{B}(x,R),r).$$
They take values on $[0,+\infty]$.
By \eqref{packing-covering} it holds
\begin{equation}
	\label{packing-covering-2}
	\text{Pack}(R,2r)\leq \text{Cov}(R,2r)\leq \text{Pack}(R,r).
\end{equation}
Let $C_0,P_0,r_0>0$. We say that a metric space $X$ is {\em $P_0$-packed at scale $r_0$} if Pack$(3r_0,r_0)\leq P_0$, that is every ball of radius $3r_0$ contains no more than $P_0$ points that are $2r_0$-separated. The space $X$ is {\em $C_0$-covered at scale $r_0$} if Cov$(3r_0,r_0)\leq C_0$, that is every ball of radius $3 r_0$ can be covered by at most $C_0$ balls of radius $r_0$. 
\vspace{2mm}

\noindent A geodesic bicombing on a metric space $X$ is a map $\sigma \colon X\times X\times [0,1] \to X$ with the property that for all $(x,y) \in X\times X$ the map $\sigma_{xy}\colon t\mapsto \sigma(x,y,t)$ is a geodesic from $x$ to $y$ parametrized proportionally to arc-length, i.e. $d(\sigma_{xy}(t), \sigma_{xy}(t')) = \vert t - t' \vert d(x,y)$ for all $t,t'\in [0,1]$, $\sigma_{xy}(0)=x, \sigma_{xy}(1)=y$. \\
\emph{When $X$ is equipped with a geodesic bicombing then for all $x,y\in X$ we will denote by $[x,y]$ the geodesic $\sigma_{xy}$ parametrized by arc-length.}\\
A geodesic bicombing is:
\begin{itemize}
	\item \emph{convex} if the map $t\mapsto d(\sigma_{xy}(t), \sigma_{x'y'}(t))$ is convex on $[0,1]$ for all $x,y,x',y' \in X$;
	\item \emph{consistent} if for all $x,y \in X$, for all $0\leq s\leq t \leq 1$ and for all $\lambda\in [0,1]$ it holds $\sigma_{pq}(\lambda) = \sigma_{xy}((1-\lambda)s + \lambda t)$, where $p:= \sigma_{xy}(s)$ and $q:=\sigma_{xy}(t)$;
	\item \emph{reversible} if $\sigma_{xy}(t) = \sigma_{yx}(1-t)$ for all $t\in [0,1]$.
\end{itemize}
For instance every convex metric space in the sense of Busemann (so also any CAT$(0)$ metric space) admits a unique convex, consistent, reversible geodesic bicombing.\\
Given a geodesic bicombing $\sigma$ we say that a geodesic (segment, ray, line) $\gamma$ is a $\sigma$-geodesic (segment, ray, line) if for all $x,y\in \gamma$ we have that $[x,y]$ coincides with the subsegment of $\gamma$ between $x$ and $y$.\\
A geodesic bicombing is \emph{geodesically complete} if every $\sigma$-geodesic segment is contained in a $\sigma$-geodesic line. 
A couple $(X,\sigma)$ is said a GCB-space if $\sigma$ is a convex, consistent, reversible, geodesically complete geodesic bicombing on the complete metric space $X$.
The packing condition has a controlled behaviour in GCB-spaces.
\begin{prop}[Proposition 3.2 of \cite{CavS20bis}]
	\label{packingsmallscales}
	Let $(X,\sigma)$ be a \textup{GCB}-space that is $P_0$-packed at scale $r_0$. Then:
	\begin{itemize}
		\item[(i)] for all $r\leq r_0$, the space $X$ is $P_0$-packed at scale $r$ and is proper;
		\item[(ii)]  for every $0<r\leq R$ and  every $x\in X$ it holds:
		\begin{equation*}
			\begin{aligned}
				\textup{Pack}(R,r)&\leq P_0(1+P_0)^{\frac{R}{r} - 1} \text{, if } r\leq r_0;\\
				\textup{Pack}(R,r)&\leq P_0(1+P_0)^{\frac{R}{r_0} - 1} \text{, if } r > r_0;\\
				\textup{Cov}(R,r)&\leq P_0(1+P_0)^{\frac{2R}{r} - 1} \text{, if } r\leq 2r_0;\\
				\textup{Cov}(R,r)&\leq P_0(1+P_0)^{\frac{R}{r_0} - 1} \text{, if } r > 2r_0.
			\end{aligned}
		\end{equation*}
	\end{itemize}  
\end{prop}


\noindent Basic examples of GCB-spaces that are $P_0$-packed at scale $r_0$ are:
\begin{itemize}
	\item[i)] complete and simply connected Riemannian manifolds with sectional curvature pinched between two nonpositive constants $\kappa' \leq  \kappa < 0$;
	\item[ii)] simply connected $M^\kappa$-complexes, with $\kappa \leq 0$, without free faces and {\em bounded geometry} (i.e., with valency at most $V_0$, size at most $S_0$ and positive injectivity radius);
	\item[iii)] complete, geodesically complete, CAT$(0)$ metric spaces $X$ 
	with dimension at most $n$ and volume of balls of radius $R_0$ bounded above by $V$.
\end{itemize} 
\noindent For further details on the second and the third class of examples we refer to \cite{CavS20}.
\vspace{2mm}

\noindent When $(X,\sigma)$ is a proper GCB-space we can consider the {\em space of parametrized geodesic lines} of $X$,
$$\text{Geod}(X) = \lbrace \gamma\colon \mathbb{R} \to X \text{ isometry}\rbrace,$$
endowed with the topology of uniform convergence on compact subsets of $\mathbb{R}$, and its subset $\text{Geod}_\sigma(X)$ made of elements whose image is a $\sigma$-geodesic line. By the continuity of $\sigma$ (due to its convexity, see \cite{Des15}, \cite{CavS20bis}) we have that $\text{Geod}_\sigma(X)$ is closed in $\text{Geod}(X)$.
There is a natural action of $\mathbb{R}$ on $\text{Geod}(X)$ defined by reparametrization:
$$\Phi_t\gamma (\cdot) = \gamma(\cdot + t)$$
for every $t\in \mathbb{R}$.
It is easy to see it is a continuous action, i.e. $\Phi_t \circ \Phi_s = \Phi_{t+s}$ for all $t,s\in \mathbb{R}$ and for every $t\in \mathbb{R}$ the map $\Phi_t$ is a homeomorphism of $\text{Geod}(X)$. Moreover the action restricts as an action on $\text{Geod}_\sigma(X)$.
This action on $\text{Geod}_\sigma(X)$ is called the {\em $\sigma$-geodesic flow} on $X$. 
The {\em evaluation map} $E\colon \text{Geod}(X) \to X$, which is defined as $E(\gamma)=\gamma(0)$, is continuous and proper (\cite{BL12}, Lemma 1.10), so its restriction to $\text{Geod}_\sigma(X)$ has the same properties. Moreover this restriction is surjective since $\sigma$ is assumed geodesically complete. The topology on $\text{Geod}_\sigma(X)$ is metrizable. Indeed we can construct a family of metrics on $\text{Geod}_\sigma(X)$ with the following method. \\
\noindent Let $\mathcal{F}$ be the class of continuous functions $f\colon \mathbb{R} \to \mathbb{R}$ satisfying
\begin{itemize}
	\item[(a)] $f(s) > 0$ for all $s \in \mathbb{R}$;
	\item[(b)] $f(s) = f(-s)$ for all $s \in \mathbb{R}$;
	\item[(c)] $\int_{-\infty}^{+\infty} f(s)ds = 1$;
	\item[(d)] $\int_{-\infty}^{+\infty} 2\vert s \vert f(s) ds = C(f) < + \infty$.
\end{itemize}
For every $f\in \mathcal{F}$ we define the distance on $\text{Geod}_\sigma(X)$:
\begin{equation}
	\label{fdistance}
	f(\gamma, \gamma') = \int_{-\infty}^{+\infty}d(\gamma(s),\gamma'(s))f(s)ds.
\end{equation}
We remark that the choice of $f=\frac{1}{2e^{\vert s \vert}}$ gives exactly the distance $d_{\text{Geod}}$. We are motivated to study the whole class $\mathcal{F}$ because of the applications in further works as \cite{Cav21bis}.
\begin{lemma}
	\label{PropGeod}
	The expression defined in \eqref{fdistance} satisfies these properties:
	\begin{itemize}
		\item[(i)] it is a well defined distance on $\textup{Geod}_\sigma(X)$;
		\item[(ii)] for all $\gamma,\gamma' \in \textup{Geod}_\sigma(X)$ it holds $f(\gamma,\gamma')\leq d(\gamma(0), \gamma(0)) + C(f)$;
		\item[(iii)] for all $\gamma,\gamma' \in \textup{Geod}_\sigma(X)$ it holds $d(\gamma(0),\gamma'(0)) \leq f(\gamma,\gamma')$;
		\item[(iv)] it induces the topology of \textup{Geod}$_\sigma(X)$.
	\end{itemize}
\end{lemma}
\begin{proof}
	For all $\gamma,\gamma'\in \text{Geod}_\sigma(X)$ we have 
	\begin{equation*}
		\begin{aligned}
			d(\gamma(s),\gamma'(s))	&\leq d(\gamma(s),\gamma(0)) + d(\gamma(0),\gamma'(0)) + d(\gamma'(0),\gamma'(s)) \\
			&\leq 2\vert s \vert + d(\gamma(0),\gamma'(0)),
		\end{aligned}
	\end{equation*}
	so
	$$\int_{-\infty}^{+\infty}d(\gamma(s),\gamma'(s))f(s)ds \leq d(\gamma(0),\gamma'(0)) + \int_{-\infty}^{+\infty}2\vert s \vert f(s) dt < +\infty.$$
	This shows (ii) and that the integral in \eqref{fdistance} is finite. From the properties of the integral and the positiveness of $f$ it is easy to prove that \eqref{fdistance} defines a distance.
	The proof of (iii) follows from the convexity of $\sigma$ and the symmetry of $f$. Indeed for all $\gamma, \gamma' \in \text{Geod}_\sigma(X)$ the function $g(s) = d(\gamma(s),\gamma'(s))$ is convex. This means that for all $S,S' \in \mathbb{R}$ and for all $\lambda \in [0,1]$ it holds
	$$g(\lambda S + (1-\lambda) S')\leq \lambda g(S) + (1-\lambda) g(S').$$
	We take $s\geq 0$ and we use the inequality above with $S=s,S'=-s$ and $\lambda = \frac{1}{2}$, obtaining
	$$d(\gamma(0),\gamma'(0))=g(0) \leq \frac{1}{2}g(-s) + \frac{1}{2}g(s) = \frac{d(\gamma(s),\gamma'(s)) + d(\gamma(-s), \gamma'(-s))}{2}.$$
	We can now estimate the distance between $\gamma$ and $\gamma'$ as
	\begin{equation*}
		\begin{aligned}
			f(\gamma, \gamma') & =\int_{-\infty}^0 d(\gamma(s),\gamma'(s))f(s)ds + \int_{0}^{+\infty}d(\gamma(s),\gamma'(s))f(s)ds \\
			&= \int_{0}^{+\infty}\big(d(\gamma(-s),\gamma'(-s)) + d(\gamma(s),\gamma'(s))\big)f(s)ds \geq d(\gamma(0),\gamma'(0)),
		\end{aligned}
	\end{equation*}
	where we used the symmetry of $f$.
	This concludes the proof of (iii). \\
	If a sequence $\gamma_n$ converges to $\gamma_\infty$ uniformly on compact subsets then it is clear that for every $T\geq 0$ it holds
	$$\lim_{n\to +\infty}\int_{-T}^{+T}d(\gamma_n(s),\gamma_\infty(s))f(s)ds = 0.$$ 
	For every $\varepsilon > 0$ we pick $T_\varepsilon \geq 0$ such that $\int_{T_\varepsilon}^{+\infty}2\vert s\vert f(s) < \varepsilon$. Then it is easy to conclude, using the properties of $f$, that $$\lim_{n\to +\infty}\int_{-\infty}^{+\infty}d(\gamma_n(s),\gamma_\infty(s))f(s)ds\leq 2\varepsilon.$$ 
	By the arbitrariness of $\varepsilon$ we conclude that the sequence $\gamma_n$ converges to $\gamma_\infty$ with respect to the metric $f$. \\
	Now suppose the sequence $\gamma_n$ converges to $\gamma_\infty$ with respect to $f$ and suppose it does not converge uniformly on compact subsets to $\gamma_\infty$. Therefore there exists $T\geq 0$, $\varepsilon_0>0$ and a subsequence $\gamma_{n_j}$ such that $d(\gamma_{n_j}(t_j), \gamma_\infty(t_j))>6\varepsilon_0$ for every $j$, where $t_j\in [-T,T]$. We can suppose $t_j\to t_\infty$ and so $d(\gamma_{n_j}(t_\infty), \gamma_\infty(t_\infty))>4\varepsilon_0$ for every $j$. For all $t\in [t_\infty - \varepsilon_0, t_\infty + \varepsilon_0]$ we get $d(\gamma_{n_j}(t), \gamma_\infty(t))>2\varepsilon_0$. Therefore, if we set $m = \min_{t\in [t_\infty - \varepsilon_0, t_\infty + \varepsilon_0]}f(s) > 0$, we obtain
	$$\int_{-\infty}^{+\infty}d(\gamma_{n_j}(s), \gamma_\infty(s))f(s)ds > 4\varepsilon_0^2m$$
	for every $j$, which is a contradiction.
\end{proof}
\noindent A metric \texthtd\, on $\text{Geod}_\sigma(X)$ inducing the topology of uniform convergence on compact subsets is said {\em geometric} if the evaluation map $E$ is Lipschitz with respect to this metric. Any metric induced by $f\in \mathcal{F}$ is geometric by Lemma \ref{PropGeod}.(iii).

\section{Covering and volume entropy}
\label{sec-covering}

In this section we will introduce the first two types of entropy: the covering entropy, defined in terms of the covering functions, and the volume entropy of a measure.

\subsection{Covering entropy}
\label{subsec-covering-basis}

Let $(X,\sigma)$ be a GCB-space that is $P_0$-packed at scale $r_0$.
It is natural to define the {\em upper covering entropy} of $X$ as the number
$$\overline{h_{\text{Cov}}}(X) = \limsup_{T\to + \infty}\frac{1}{T}\log \text{Cov}(\overline{B}(x,T),r_0),$$
where $x$ is any point of $X$.
The {\em lower covering entropy} is defined taking the limit inferior instead of the limit superior and it is denoted by $\underline{h_{\text{Cov}}}(X)$. \linebreak
By triangular inequality it is easy to show that the definitions of upper and lower covering entropy do not depend on the point $x\in X$.
In the next proposition, which is essentially Proposition 3.2 of \cite{Cav21ter}, we can see that they do not depend on $r_0$ too and moreover we can replace the covering function with the packing function.
\begin{prop}
	\label{cov-ent-pack}
	Let $(X,\sigma)$ be a \textup{GCB}-space that is $P_0$-packed at scale $r_0$ and let $x\in X$. Then
	$$\frac{1}{T}\log \textup{Cov}(\overline{B}(x,T),r) \underset{P_0,r_0,r,r'}{\asymp} \frac{1}{T}\log \textup{Pack}(\overline{B}(x,T),r')$$ for all $r, r'> 0$.
	In particular any of these functions can be used in the definition of the upper and lower covering entropy.
\end{prop}
\begin{proof}
	For all $0<r \leq r'$ and $x\in X$ clearly $\text{Cov}(\overline{B}(x,T), r) \geq \text{Cov}(\overline{B}(x,T),r')$
	and $\text{Cov}(\overline{B}(x,T), r)\leq \text{Cov}(\overline{B}(x,T),r')\cdot \sup_{y\in X}\text{Cov}(\overline{B}(y,r'),r).$
	By Proposition \ref{packingsmallscales} we have $\sup_{y\in X}\text{Cov}(\overline{B}(y,r'),r) = \text{Cov}(r',r)$ which is a finite number depending only on $P_0,r_0, r,r'$. Therefore we obtain 
	$$\frac{1}{T}\log \text{Cov}(\overline{B}(x,T),r) \underset{P_0,r_0,r, r'}{\asymp} \frac{1}{T}\log \text{Cov}(\overline{B}(x,T),r').$$
	Now the thesis follows from \eqref{packing-covering}.
\end{proof}
The upper and lower covering entropies can be computed also using the covering function of the metric spheres.
\begin{prop}
	\label{spherical-entropy}
	Let $(X,\sigma)$ be a \textup{GCB}-space that is $P_0$-packed at scale $r_0$ and $x\in X$. Then for all $r > 0$
	$$\frac{1}{T}\log \textup{Cov}(\overline{B}(x,T),r)\underset{P_0,r_0,r}{\asymp}\frac{1}{T}\log \textup{Cov}(S(x,T),r)$$
\end{prop}
\begin{proof}
	Clearly it holds $\text{Cov}(S(x,T),r) \leq \text{Cov}(\overline{B}(x,T),r).$
	The other estimate is more involved. We divide the ball $\overline{B}(x,T)$ in annulii $A(x,kr, (k+1)r)$ with $k=0, \ldots, \frac{T}{r}-1$.
	We easily obtain
	$$\text{Cov}(\overline{B}(x,T),2r) \leq \sum_{k=0}^{\frac{T}{r}-1}\text{Cov}(A(x,kr,(k+1)r), 2r).$$
	Now we claim that for any $k$ it holds 
	$$\text{Cov}(A(x,kr,(k+1)r), 2r) \leq \text{Cov}(S(x,T),r).$$
	Indeed let $\lbrace y_1,\ldots, y_N \rbrace$ be a set of points realizing $\text{Cov}(S(x,T),r)$. For all $i=1,\ldots,N$ we consider the $\sigma$-geodesic segment $\gamma_i=[x,y_i]$ and we call $x_i$ the point along this geodesic segment at distance $kr$ from $x$. Then $x_i \in A(x,kr,(k+1)r)$ for every $i=1,\ldots, N$. We claim that $\lbrace x_1,\dots,x_N \rbrace$ is a $2r$-dense subset of $A(x,kr,(k+1)r)$. We take any $y \in A(x,kr,(k+1)r)$ and we consider the $\sigma$-geodesic segment $[x,y]$. We extend this geodesic to a $\sigma$-geodesic segment $\gamma=[x,y']$, where $y'$ is at distance $T$ from $x$. Then there exists $i$ such that $d(y', y_i) = d(\gamma(T),\gamma_i(T))\leq r$. By convexity of $\sigma$ we have $d(\gamma(t),\gamma_i(t))\leq r$, where $t=d(x,y)$. Therefore we conclude that $d(y,x_i) \leq d(y,\gamma_i(t)) + d(\gamma_i(t), x_i) \leq 2r.$
	This ends the proof of the claim, so
	$\text{Cov}(\overline{B}(x,T),2r) \leq \frac{T}{r}\text{Cov}(S(x,T),r).$
	The thesis follows from these estimates and Proposition \ref{cov-ent-pack}.
\end{proof}
Combining Proposition \ref{packingsmallscales} and Proposition \ref{cov-ent-pack} we can find an uniform upper bound to the covering entropy, see also Proposition 3.2 of \cite{Cav21ter}.
\begin{lemma}
	\label{entropy-bound}
	Let $(X,\sigma)$ be a \textup{GCB}-space that is $P_0$-packed at scale $r_0$. Then
	$$\overline{h_{\textup{Cov}}}(X) \leq \frac{\log(1+P_0)}{r_0}.$$
\end{lemma}
\begin{proof}
	For every $x \in X$ it holds
	$\text{Pack}(\overline{B}(x,R),r_0) \leq P_0(1+P_0)^{\frac{R}{r_0} - 1}.$
	The thesis follows immediately.
\end{proof}

\subsection{Volume entropy of homogeneous measures}
\label{subsec-volume}
Let $(X,\sigma)$ be a \textup{GCB}-space that is $P_0$-packed at scale $r_0$. The {\em upper volume entropy} of a measure $\mu$ on $X$ is defined as
$$\overline{h_\mu}(X) = \limsup_{T\to + \infty}\frac{1}{T}\log \mu(\overline{B}(x,T)),$$
while the {\em lower volume entropy} $\underline{h_\mu}(X)$ is defined taking the limit inferior. These definitions do not depend on the choice of the point $x\in X$.\\
A measure $\mu$ on $X$ is called {\em $H$-homogeneous at scale $r > 0$} if
$$\frac{1}{H} \leq \mu(\overline{B}(x,r)) \leq H$$
for all $x\in X$.
We remark that the condition must hold only at scale $r$.
\begin{prop}
	\label{cov-vol-asym}
	Let $(X,\sigma)$ be a \textup{GCB}-space that is $P_0$-packed at scale $r_0$ and let $\mu$ be a measure on $X$ which is $H$-homogeneous at scale $r$. Then
	$$\frac{1}{T}\log \mu(\overline{B}(x,T))\underset{P_0,r_0,H,r}{\asymp} \frac{1}{T}\log \textup{Cov}(\overline{B}(x,T), r).$$
	In particular the upper (resp. lower) volume entropy of $\mu$ coincides with the upper (resp. lower) covering entropy of $X$.
\end{prop}
\begin{proof}
	For all $x\in X$ it holds $\mu(\overline{B}(x,T)) \leq H \cdot \text{Cov}(\overline{B}(x,T), r)$
	and \linebreak
	$\mu(\overline{B}(x,T)) \geq \frac{1}{H} \cdot  \text{Pack}(\overline{B}(x,T - r), r).$\\
	By Proposition \ref{cov-ent-pack} and since $\frac{T-r}{T}\underset{r}{\asymp} 1$ we have the thesis.
\end{proof}
\begin{obs}
	\label{homogeneous}
	The proof of the proposition shows another fact: if a measure is $H$-homogeneous at scale $r$ then it is $H(r')$-homogeneous at scale $r'$ for all $r'\geq r$ and $H(r')$ depends only on $H, P_0, r_0, r$ and $r'$.
\end{obs} 

We provide here an example of homogeneous measure.
If $X$ is a complete, geodesically complete, CAT$(0)$ metric space that is $P_0$-packed at scale $r_0$ then the natural measure on $X$ satisfies
$$c \leq \mu_X(\overline{B}(x,r_0)) \leq C$$
for all $x \in X$, where $c$ and $C$ are constants depending only on $P_0$ and $r_0$ (Theorem 4.9 of \cite{CavS20}). It follows immediately the following result.
\begin{cor}
	Let $X$ be a complete, geodesically complete, \textup{CAT}$(0)$ metric space. If it is $P_0$-packed at scale $r_0$ for some $P_0$ and $r_0$ then $\overline{h_{\textup{Cov}}}(X) = \overline{h_{\mu_X}}(X).$
	The same holds for the lower entropies.
\end{cor}

\section{Lipschitz-topological entropy}
\label{sec-liptop}
Let $(X,\sigma)$ be a \textup{GCB}-space that is $P_0$-packed at scale $r_0$. The space $\text{Geod}_\sigma(X)$ is locally compact but not compact. The topological entropy of  the geodesic flow can be defined (see \cite{Bow73}, \cite{HKR95}) as
$$\overline{h_{\text{top}}}(\text{Geod}_\sigma(X)) = \inf_{\text{\texthtd}} \sup_{K} \lim_{r\to 0} \limsup_{T\to + \infty} \frac{1}{T}\log \text{Cov}_{\text{\texthtd}^T}(K,r),$$
where the infimum is taken among all metrics \texthtd\, inducing the topology of $\text{Geod}_\sigma(X)$, the supremum is taken among all compact subsets of $\text{Geod}_\sigma(X)$ and $\text{Cov}_{\text{\texthtd}^T}(K,r)$ is the covering function of the compact subset $K$ at scale $r$ with respect to the metric \texthtd$^T$ defined by
$$\text{\texthtd}^T(\gamma,\gamma') = \max_{t\in [0,T]} \text{\texthtd}(\Phi_t\gamma, \Phi_t\gamma').$$
By the variational principle this quantity equals the measure-theoretic entropy defined as the supremum of the entropies of the flow-invariant probability measures on $\text{Geod}_\sigma(X)$ (cp. \cite{HKR95}, Lemma 1.5).  
An easy computation shows that the topological entropy is always zero.
\begin{lemma}
	\label{zero-entropy}
	There are no flow-invariant probability measures on $\textup{Geod}_\sigma(X)$. In particular the topological entropy of the geodesic flow is $0$.
\end{lemma}
\begin{proof}
	Suppose there is a flow-invariant probability measure $\mu$ on $\text{Geod}_\sigma(X)$. For $x\in X$ and $R\geq 0$ we define $A_R = \lbrace \gamma \in \text{Geod}_\sigma(X) \text{ s.t. } \gamma(0)\in \overline{B}(x,R) \rbrace.$ 
	Clearly	there exists $R\geq 0$ such that $\mu(A_R)> \frac{1}{2}$. By flow-invariance of $\mu$ we have that the set
	$$\Phi_{2R + 1}^{-1}(A_R) = \lbrace \gamma \in \text{Geod}_\sigma(X) \text{ s.t. } \gamma(2R + 1) \in \overline{B}(x,R)\rbrace$$
	has measure $>\frac{1}{2}$. This implies that $\mu(A_R \cap \Phi_{2R + 1}^{-1}(A_R)) > 0$, but this intersection is empty.
\end{proof}

Looking at the proof of the variational principle given in \cite{HKR95} we can observe that the sequence of metrics on $\text{Geod}_\sigma(X)$ that approach the infimum in the definition of the topological entropy are the restriction to $\text{Geod}_\sigma(X)$ of metrics defined on its one-point compactification. 
These metrics are not the natural ones on $\text{Geod}_\sigma(X)$, since they are not geometric. We propose a more appropriate definition of topological entropy for proper GCB-spaces.\\
We define the {\em upper Lipschitz-topological entropy} of $\text{Geod}_\sigma(X)$ as
$$\overline{h_{\text{Lip-top}}}(\text{Geod}_\sigma(X)) = \inf_{\text{\texthtd}}\sup_K \lim_{r\to 0} \limsup_{T\to + \infty}\frac{1}{T}\log\text{Cov}_{\text{\texthtd}^T}(K,r),$$
where the infimum is now taken only among all geometric metrics on $\text{Geod}_\sigma(X)$. The {\em lower Lipschitz-topological entropy} is defined by taking the limit inferior instead of the limit superior and it is denoted by $\underline{h_{\text{Lip-top}}}(\text{Geod}_\sigma(X))$. The main result of this section is the following.

\begin{theo}
	\label{lip-top_cov}
	Let $(X,\sigma)$ be a \textup{GCB}-space that is $P_0$-packed at scale $r_0$. Then
	$$\overline{h_{\textup{Lip-top}}}(\textup{Geod}_\sigma(X)) = \overline{h_{\textup{Cov}}}(X).$$
	The same holds for the lower entropies.
\end{theo}
\noindent One of the two inequalities is easy. In order to prove the other one we will show that for the distances induced by the functions $f\in \mathcal{F}$ the definition of topological entropy can be heavily simplified.

\subsection{Topological entropy for the distances induced by $f\in \mathcal{F}$}
\label{subsec-top-f}

For a metric $f\in \mathcal{F}$ we denote by $\overline{h_f}$ the upper metric entropy of the $\sigma$-geodesic flow with respect to $f$, that is
$$\overline{h_f}(\text{Geod}_\sigma(X))=\sup_K \lim_{r\to 0} \limsup_{T\to + \infty}\frac{1}{T}\log\text{Cov}_{f^T}(K,r).$$
In the usual way it is defined the lower metric entropy with respect to $f$, $\underline{h_f}(\text{Geod}(X))$. For a subset $Y$ of $X$ we denote by $\text{Geod}_\sigma(Y)$ the set of $\sigma$-geodesic lines of $X$ passing through $Y$ at time $0$.

\begin{prop}
	\label{entropyrelations}
	Let $(X,\sigma)$ be a \textup{GCB}-space that is $P_0$-packed at scale $r_0$ and let $f\in \mathcal{F}$. Then
	\begin{itemize}
		\item[(i)] for all $x,y \in X$ it holds
		$\overline{h_f}(\textup{Geod}_\sigma(x))=\overline{h_f}(\textup{Geod}_\sigma(y));$
		\item[(ii)] for all $x \in X$ and $R\geq 0$ it holds 
		$\overline{h_f}(\textup{Geod}_\sigma(\overline{B}(x,R))) = \overline{h_f}(\textup{Geod}_\sigma(x));$
		\item[(iii)] for all $x\in X$ it holds
		$ \overline{h_f}(\textup{Geod}_\sigma(X)) = \overline{h_f}(\textup{Geod}_\sigma(x)) \leq \overline{h_{\textup{Cov}}}(X)$;
		\item[(iv)] for all $x\in X$ the function
		$r\mapsto \limsup_{T\to +\infty}\frac{1}{T}\log \textup{Cov}_{f^T}(\textup{Geod}_\sigma(x), r)$
		is constant. 
	\end{itemize}
	The same conclusions hold for the lower Lipschitz-topological entropy.
\end{prop}
\noindent The proposition is consequence of the following fundamental key lemma.
\begin{lemma}[Key Lemma]
	\label{residualentropy}
	Let $f\in \mathcal{F}$, $\gamma \in \textup{Geod}_\sigma(X)$ and $0<r\leq r'$. Then
	$$\frac{1}{T}\log \textup{Cov}_{f^T}(\overline{B}_{f^T}(\gamma,r'), r) \underset{P_0,r_0,r,r',f}{\asymp} 0,$$
	where $\overline{B}_{f^T}(\gamma,r')$ is the closed ball of center $\gamma$ and radius $r'$ with respect to the metric $f^T$.
	We remark that the convergence is uniform in $\gamma$.
\end{lemma}
\begin{proof}
	Let $P>0$ depending only on $f$ and $r'$ such that
	$$\int_{-\infty}^{-P} 2 \vert u \vert f(u) du + \int_P^{+\infty} 2\vert u \vert f(u)du < \frac{r}{4}.$$
	We fix $\varepsilon > 0$ and $T\geq \frac{P}{\varepsilon}$. Let $E_T = \lbrace x_1, \ldots, x_N \rbrace$ be a maximal $\frac{r}{16}$-separated subset of $B(\gamma(T), r' + \varepsilon T)$, so it is also $\frac{r}{16}$-dense, and $\lbrace y_1,\ldots, y_M \rbrace$ be a $\frac{r}{16}$-dense subset of $B(\gamma(-P),r' + 2P)$. 
	For every $i=1,\ldots,M$ and $j=1,\ldots, N$ we take a $\sigma$-geodesic line $\gamma_{ij}$ extending the $\sigma$-geodesic segment $[y_i,x_j]$. We parametrize $\gamma_{ij}$ in such a way that $\gamma_{ij}(-P)=y_i$. The claim is that $\lbrace \gamma_{ij} \rbrace_{i,j}$ is a $r$-dense subset of $\overline{B}_{f^T}(\gamma, r')$ with respect to the metric $f^T$.
	We fix $\gamma' \in \overline{B}_{f^T}(\gamma, r')$. This means 
	$$\max_{t\in [0,T]} f^t(\gamma', \gamma) = \max_{t\in [0,T]} f(\Phi_{t}(\gamma'), \Phi_{t}(\gamma)) \leq r'.$$
	In particular for all $t\in [0,T]$ we get $d(\gamma'(t), \gamma(t))\leq r'$, since 
	$$d(\gamma'(t),\gamma(t)) =  d(\Phi_t(\gamma'),\Phi_t(\gamma)) \leq f(\Phi_t(\gamma'),\Phi_t(\gamma)) \leq r'.$$
	Therefore $d(\gamma'(-P), \gamma(-P)) \leq r' + 2P.$
	Moreover
	$$d(\gamma'(T + \varepsilon T), \gamma(T)) \leq d(\gamma'(T + \varepsilon T), \gamma'(T)) + d(\gamma'(T), \gamma(T)) \leq \varepsilon T + r'.$$
	Thus there exists $x_j$ such that $d(x_j, \gamma'(T + \varepsilon T)) \leq \frac{r}{16}$ and $y_i$ such that $d(y_i, \gamma'(-P)) \leq \frac{r}{16}$.
	We have $d(\gamma_{ij}(-P), \gamma'(-P)) \leq \frac{r}{16}$, so if we denote with $t_j$ the time such that $\gamma_{ij}(t_j)=x_j$ it holds $\vert t_j - (T + \varepsilon T) \vert \leq \frac{r}{8}$. Then
	\begin{equation*}
		\begin{aligned}
			d(\gamma_{ij}(T + \varepsilon T), \gamma'(T + \varepsilon T)) &\leq d(\gamma_{ij}(T + \varepsilon T), \gamma_{ij}(t_j)) + d(\gamma_{ij}(t_j), \gamma'(T + \varepsilon T)) \\
			&\leq \frac{r}{8} + \frac{r}{16} < \frac{r}{4}.
		\end{aligned}
	\end{equation*}
	From the convexity of $\sigma$ we have $d(\gamma'(u), \gamma_{ij}(u)) < \frac{r}{4}$ for all $u\in [-P, (1+\varepsilon)T]$.
	For $t\in [0,T]$ we have
	\begin{equation*}
		\begin{aligned}
			f^t(\gamma', \gamma_{ij}) &= \int_{-\infty}^{+\infty} d(\gamma'(u), \gamma_{ij}(u)) f( u - t) du \\
			&\leq\int_{-\infty}^{-P}\bigg(\frac{r}{4} + 2\vert u + P \vert \bigg) f( u - t) du + \\
			&+\int_{-P}^{ (1+\varepsilon)T} \frac{r}{4} f(u - t)du +\\
			&+ \int_{(1+\varepsilon)T}^{+\infty} \bigg(\frac{r}{4} + 2\vert u - (1+\varepsilon)T \vert \bigg) f(u - t) du.
		\end{aligned}
	\end{equation*}
	The first term can be estimated as follows
	\begin{equation*}
		\begin{aligned}
			\int_{-\infty}^{-P}\bigg(\frac{r}{4} + 2\vert u + P \vert \bigg) f( u - t) du &\leq \frac{r}{4} + \int_{-\infty}^{-P-t} 2\vert v + t + P \vert f(v) dv \\
			&\leq \frac{r}{4} + \int_{-\infty}^{-P} 2\vert v \vert f(v) dv.
		\end{aligned}
	\end{equation*}
	The second term is less than or equal to $\frac{r}{4}$. The third term can be controlled in this way:
	\begin{equation*}
		\begin{aligned}
			\int_{(1+\varepsilon)T}^{+\infty} \bigg(\frac{r}{4} + 2\vert u - (1+\varepsilon)T \vert\bigg) f(u - t) du &\leq \frac{r}{4} + \int_{(1+\varepsilon)T -t}^{+\infty} 2 \vert v - (1+\varepsilon)T + t \vert f(v) dv \\
			&\leq \frac{r}{4} + \int_{(1+\varepsilon)T -t}^{+\infty} 2 \vert v \vert f(v) dv \\
			&\leq \frac{r}{4} + \int_{P}^{+\infty} 2 \vert v \vert f(v) dv.
		\end{aligned}
	\end{equation*}
	The last inequality follows from $T\geq \frac{P}{\varepsilon}$.
	Therefore 
	$$f^t(\gamma', \gamma_{ij}) \leq \frac{r}{4} + \frac{r}{4} + \frac{r}{4} + \int_{-\infty}^{-P} 2 \vert v \vert f(v) dv + \int_{P}^{+\infty} 2 \vert v \vert f(v) dv \leq r.$$
	We conclude that 
	$$\text{Cov}_{f^T}(\overline{B}_{f^T}(\gamma, r'), r) \leq \text{Cov}\bigg(r'+2P, \frac{r}{16}\bigg) \cdot \# E_T.$$
	From Proposition \ref{packingsmallscales}, if $\rho = \min\left\lbrace r_0, \frac{r}{16}\right\rbrace$, we get $\# E_T \leq P_0(1+P_0)^{\frac{r'+\varepsilon T}{\rho} - 1}.$
	Thus
	\begin{equation*}
		\begin{aligned}
			\frac{1}{T}\log \text{Cov}_{f^T}(\overline{B}_{f^T}(\gamma,r'), r) &\leq 
			\frac{1}{T} K(P_0,r_0,r,r',f)\cdot \frac{\varepsilon T}{\rho}\log(1+P_0) \\
			&= \varepsilon\cdot K'(P_0,r_0,r,r',f).
		\end{aligned}
	\end{equation*}
	Here $K, K'$ are constants depending only on $P_0,r_0,r,r',f$ and not on $\varepsilon$ or $\gamma$.
	So from the arbitrariness of $\varepsilon$ we achieve the proof.
\end{proof}

The computation of $\overline{h_f}$ requires to consider the supremum among all compact subsets of $\text{Geod}_\sigma(X)$. We notice that given a compact subset $K\subseteq \text{Geod}_\sigma(X)$, the set $E(K)$ is compact since $E$ is continuous. In particular it is bounded, hence contained in a ball $\overline{B}(x,R)$ centered at a reference point $x \in X$. We observe also that the set $\text{Geod}_\sigma(\overline{B}(x,R))$ is compact since the evaluation map $E$ is proper. We conclude that any compact subset of $\text{Geod}_\sigma(X)$ is contained in a compact subset of the form $\text{Geod}_\sigma(\overline{B}(x,R))$ and therefore in order to compute $\overline{h_f}$ it is enough to take the supremum among these sets.
The main consequence of Lemma \ref{residualentropy} is the following result, which is the key ingredient in the proof of Proposition \ref{entropyrelations}.
\begin{cor}
	\label{entropyrescale}
	Let $f\in \mathcal{F}$, $x\in X$, $R\geq 0$ and $0 < r \leq r'$. Then
	$$\frac{1}{T}\log \textup{Cov}_{f^T}(\textup{Geod}_\sigma(\overline{B}(x,R)),r) \underset{P_0,r_0,r,r',f}{\asymp} \frac{1}{T} \log \textup{Cov}_{f^T}(\textup{Geod}_\sigma(\overline{B}(x,R)),r').$$
\end{cor}
\begin{proof}
	The quantity $\frac{1}{T}\log \text{Cov}_{f^T}(\text{Geod}_\sigma(\overline{B}(x,R)),r)$ is
	\begin{equation*}
		\begin{aligned}
			&\leq \frac{1}{T} \log \text{Cov}_{f^T}(\text{Geod}_\sigma(\overline{B}(x,R)), r') \cdot \sup_{\gamma \in \text{Geod}_\sigma(X)} \text{Cov}_{f^T}(\overline{B}_{f^T}(\gamma, r'), r) \\
			&=\frac{1}{T} \big(\log \text{Cov}_{f^T}(\text{Geod}_\sigma(\overline{B}(x,R)), r') + \log \sup_{\gamma \in X} \text{Cov}_{f^T}(\overline{B}_{f^T}(\gamma, r'), r) \big)
		\end{aligned}
	\end{equation*}
	The conclusion follows by Lemma \ref{residualentropy}.
\end{proof}

\begin{proof}[Proof of Proposition \ref{entropyrelations}.(ii)]
	Let $\varepsilon > 0$ and $T > \frac{R}{\varepsilon}$.
	Let $\gamma_1, \ldots, \gamma_N$ be a $r$-dense subset of Geod$_\sigma(x)$ with respect to the metric $f^{(2+\varepsilon)T}$. The claim is that $\lbrace\gamma_i \rbrace$ is a $K$-dense subset of $\text{Geod}_\sigma(\overline{B}(x,R))$ with respect to $f^T$, where $K$ depends only on $r,R$ and $f$.
	We consider a $\sigma$-geodesic line $\gamma \in \text{Geod}_\sigma(\overline{B}(x,R))$. Then there exists a $\sigma$-geodesic line $\gamma' \in \text{Geod}_\sigma(x)$ extending the $\sigma$-geodesic segment $[x,\gamma((1+\varepsilon)T)]$. We call $t_{\gamma'}$ the time such that $\gamma'(t_{\gamma'}) = \gamma((1+\varepsilon)T)$. Then
	\begin{equation*}
		\begin{aligned}
			t_{\gamma'} = d(x, \gamma((1+\varepsilon)T)) &\leq d(x,\gamma(0)) + d(\gamma(0),\gamma((1+\varepsilon)T))\\
			&\leq R + (1+\varepsilon)T \leq (1+2\varepsilon)T
		\end{aligned}
	\end{equation*}
	since $T\geq \frac{R}{\varepsilon}$. Moreover $\vert t_{\gamma'} - (1+\varepsilon)T\vert \leq R.$
	We know there exists $\gamma_i$ such that 
	$\max_{t\in [0,(1+2\varepsilon)T]}f(\Phi_t \gamma', \Phi_t \gamma_i) \leq r.$
	In particular $d(\gamma'(t_{\gamma'}), \gamma_i(t_{\gamma'}))\leq r$.
	Then $d(\gamma((1+\varepsilon)T), \gamma_i(t_{\gamma'})) \leq r$ and in conclusion
	$$d(\gamma((1+\varepsilon)T), \gamma_i((1+\varepsilon)T)) \leq d(\gamma((1+\varepsilon)T), \gamma_i(t_{\gamma'})) + d(\gamma_i(t_{\gamma'}), \gamma_i((1+\varepsilon)T)) \leq r + R.$$
	From the convexity of $\sigma$ we have $d(\gamma(t), \gamma_i(t)) \leq R + r$ for all $t\in [0,(1+\varepsilon)T]$.
	We have to estimate $f^t(\gamma, \gamma_i) = \int_{-\infty}^{+\infty} d(\gamma(u),\gamma_i(u)) f(u - t) du$ for every $t\in [0,T]$.
	Since $d(\gamma(0),\gamma_i(0))\leq R$ 
	and $d(\gamma((1+\varepsilon)T),\gamma_i((1+\varepsilon)T))\leq r + R$ then 
	\begin{equation*}
		\begin{aligned}
			\int_{-\infty}^{+\infty} d(\gamma(u),\gamma_i(u)) f(u - t) du \leq &\int_{-\infty}^0 (R+2\vert u \vert) f(u-t)du + \\
			+ &\int_0^{(1+\varepsilon)T} (R + r) f(u - t)du + \\
			+ &\int_{(1+\varepsilon)T}^{+\infty} (R + r + 2\vert u - (1+\varepsilon)T \vert) f(u - t)du 
		\end{aligned}
	\end{equation*}
	$$\leq R + \int_{-\infty}^{-t}2\vert v + t \vert f(v)dv + (R + r) + \int_{(1+\varepsilon)T - t}^{+\infty}(R + r + 2\vert v -(1+\varepsilon)T + t \vert) f(v) dv.$$
	We conclude that the above quantity is less than or equal to 
	$$3R + 2r + \int_{-\infty}^{0}2\vert v \vert f(v)dv + \int_{0}^{+\infty} 2\vert v \vert f(v)dv \leq 3R + 2r + C(f) = K(R,r,f).$$
	By the previous corollary $\overline{h_f}(\text{Geod}_\sigma(\overline{B}(x,R)))$can be computed as 
	$$\limsup_{T\to + \infty}\frac{1}{T} \log \text{Cov}_{f^T}(\text{Geod}_\sigma(\overline{B}(x,R)), K)$$
	which is
	\begin{equation*}
		\begin{aligned}
			&\leq \limsup_{T\to + \infty} \frac{1}{T} \log \text{Cov}_{f^{(1+2\varepsilon)T}}(\text{Geod}_\sigma(x),r) \\
			&= (1+2\varepsilon)\limsup_{T\to + \infty} \frac{1}{T} \log \text{Cov}_{f^T}(\text{Geod}_\sigma(x),r).
		\end{aligned}
	\end{equation*}
	Since this is true for all $\varepsilon > 0$ then we obtain the thesis.
\end{proof}

\begin{proof}[Proof of Proposition \ref{entropyrelations}.(i)]
	We have $y\in \overline{B}(x,R)$, where $R=d(x,y)$, so $\text{Geod}_\sigma(y)\subseteq \text{Geod}_\sigma(\overline{B}(x,R))$. Therefore 
	$$\overline{h}_f(\text{Geod}_\sigma(y)) \leq \overline{h}_f(\text{Geod}_\sigma(\overline{B}(x,R))) = \overline{h}_f(\text{Geod}_\sigma(x)).$$
	The other inequality can be proved in the same way.
\end{proof}

Finally we achieve the proof of the remaining parts of Proposition $\ref{entropyrelations}$.
\begin{proof}[Proof of Proposition \ref{entropyrelations}.(iii) \& (iv)]
	The equality in (iii) follows directly from (ii), so
	$$\overline{h_f}(\text{Geod}_\sigma(X)) = \limsup_{T \to +\infty}\frac{1}{T}\log \text{Cov}_{f^T}(\text{Geod}_\sigma(x),r_0),$$
	where $x$ is a point of $X$.
	We fix $T > 0$ and we consider a $r_0$-separated subset $E_T$ of $S(x_0,T)$ of maximal cardinality, which is also $r_0$-dense. For all $y \in E_T$ we consider a $\sigma$-geodesic line $\gamma_y$ extending the $\sigma$-geodesic segment $[x_0,y]$ such that $\gamma_y(0)=x_0$ and $\gamma_y(T)=y$. We claim that $\lbrace \gamma_y \rbrace_{y\in E_T}$ is a $(r_0 + C(f))$-dense subset of $\text{Geod}_\sigma(x)$ with respect to $f^T$.
	We take a $\sigma$-geodesic line $\gamma \in \text{Geod}_\sigma(x)$. Then there exists $y \in E_T$ such that $d(\gamma(T),y) = d(\gamma(T),\gamma_y(T))\leq r_0.$
	From the convexity of $\sigma$ it holds
	$d(\gamma(u),\gamma_y(u))\leq r_0$
	for all $u\in [0,T]$. Moreover $d(\gamma(u), \gamma_y(u))\leq r_0 + 2\vert u - T \vert$ for all $u \in [T,+\infty)$ and $d(\gamma(u),\gamma_y(u))\leq 2\vert u \vert $ for all $u\in (-\infty, 0]$.
	Then for all $t\in [0,T]$ we get
	\begin{equation*}
		\begin{aligned}
			f^t(\gamma, \gamma_y) &= \int_{-\infty}^{+\infty} d(\gamma(u),\gamma_y(u)) f(u - t) du \\ &\leq \int_{-\infty}^0 2\vert u \vert f(u - t) du + \int_0^{T} r_0 f(u - t)du + \\
			&+ \int_{T}^{+\infty} (r_0 + 2\vert u - T \vert) f(u - t)du \leq r_0 + C(f).
		\end{aligned}
	\end{equation*}
	The last inequality follows from similar estimates given in the proofs of Lemma \ref{residualentropy}.
	Therefore applying Corollary \ref{entropyrescale} we have
	\begin{equation*}
		\begin{aligned}
			\limsup_{T\to +\infty}\frac{1}{T}\log \text{Cov}_{f^T}(\text{Geod}_\sigma(x),r_0) \leq \limsup_{T\to +\infty}\frac{1}{T}\log\text{Cov}(S(x,T), r_0).
		\end{aligned}
	\end{equation*}
	This, together with Proposition \ref{spherical-entropy}, proves (iii).
	We observe that (iv) is exactly Corollary \ref{entropyrescale} with $R=0$.
\end{proof}

\subsection{Proof of Theorem \ref{lip-top_cov}}
\label{subsec-liptop-proof}

We are ready to give the
\begin{proof}[Proof of Theorem \ref{lip-top_cov}] Proposition \ref{entropyrelations}.(iii) shows that $\overline{h_\text{Lip-top}}(\text{Geod}_\sigma(X))$ is less than or equal to $\overline{h_\text{Cov}}(X)$. \\
	In order to prove the other inequality we fix a geometric metric \texthtd\, on $\text{Geod}_\sigma(X)$ and we denote by $M$ the Lipschitz constant with respect to \texthtd\, of the evaluation map $E$. Then we have
	$$\sup_K \lim_{r\to 0} \limsup_{T\to + \infty}\frac{1}{T}\log\text{Cov}_{\text{\texthtd}^T}(K,r) \geq \limsup_{T\to + \infty}\frac{1}{T}\log\text{Cov}_{\text{\texthtd}^T}(\text{Geod}_\sigma(x),r_0),$$
	where $x\in X.$
	We fix $T\geq 0$ and we consider a set $\gamma_1,\ldots,\gamma_N$ realizing \linebreak $\text{Cov}_{\text{\texthtd}^T}(\text{Geod}_\sigma(x),r_0)$. The claim is that $\gamma_i(T)$ is a $Mr_0$-dense subset of $S(x,T)$. Indeed we take a point $y\in S(x,T)$ and we extend the $\sigma$-geodesic segment $[x,y]$ to a $\sigma$-geodesic line $\gamma \in \text{Geod}_\sigma(x)$. Then there exists $\gamma_i$ such that \texthtd$^T(\gamma, \gamma_i)\leq r_0$. Since the evaluation map is $M$-Lipschitz we have 
	$$d(y,\gamma_i(T)) = d(\gamma(T),\gamma_i(T)) = d(\Phi_T\gamma(0), \Phi_T\gamma_i(0)) \leq L\text{\texthtd}(\Phi_T\gamma, \Phi_T\gamma_i) \leq Mr_0.$$
	Therefore
	$$\limsup_{T\to + \infty}\frac{1}{T}\log\text{Cov}_{\text{\texthtd}^T}(\text{Geod}_\sigma(x),r_0) \geq \limsup_{T\to + \infty}\frac{1}{T}\log\text{Cov}(S(x,T),Mr_0)$$
	and the conclusion follows by Proposition \ref{spherical-entropy}.
\end{proof}
\begin{obs}
	\label{asymptotic-lipschitz}
	By Proposition \ref{entropyrelations} and Theorem \ref{lip-top_cov} the upper Lipschitz-topological entropy of $X$ can be computed as
	$$\overline{h_{\textup{Lip-top}}}(X) = \limsup_{T\to +\infty}\frac{1}{T}\log \textup{Cov}_{f^T}(\textup{Geod}_\sigma(x), r)$$
	independently of $f\in \mathcal{F}$, $x\in X$ and $r>0$. Moreover 
	$$\frac{1}{T}\log \textup{Cov}_{f^T}(\textup{Geod}_\sigma(x), r_0) \underset{P_0,r_0, f}{\asymp} \frac{1}{T}\log\textup{Cov}(\overline{B}(x,T),r_0)$$
	by the proofs of Theorem \ref{lip-top_cov} and Proposition \ref{entropyrelations} and by Proposition \ref{spherical-entropy}.
\end{obs}

\section{Gromov-hyperbolic metric spaces}
In the second part of the paper we will study the versions of the entropies introduced in the first part relative to subsets of the boundary at infinity. In order to have meaningful definitions we will consider Gromov-hyperbolic metric spaces.
\vspace{2mm}

\noindent Let $X$ be a geodesic space. Given  three points  $x,y,z \in X$,  the {\em Gromov product} of $y$ and $z$ with respect to $x$  is defined as
\vspace{-3mm}

$$(y,z)_x = \frac{1}{2}\big( d(x,y) + d(x,z) - d(y,z) \big).$$

\noindent The space $X$ is said {\em $\delta$-hyperbolic} if   for every four points $x,y,z,w \in X$   the following {\em 4-points condition} hold:
\begin{equation}\label{hyperbolicity}
	(x,z)_w \geq \min\lbrace (x,y)_w, (y,z)_w \rbrace -  \delta. 
\end{equation}

\noindent The space $X$ is   {\em Gromov hyperbolic} if it is $\delta$-hyperbolic for some $\delta \geq 0$. 

\noindent Let $X$ be a proper, $\delta$-hyperbolic metric space $x$ be a point of $X$. \\
The {\em Gromov boundary} of $X$ is defined as the quotient 
$$\partial X = \lbrace (z_n)_{n \in \mathbb{N}} \subseteq X \hspace{1mm} | \hspace{1mm}   \lim_{n,m \to +\infty} (z_n,z_m)_{x} = + \infty \rbrace \hspace{1mm} /_\sim,$$
where $(z_n)_{n \in \mathbb{N}}$ is a sequence of points in $X$ and $\approx$ is the equivalence relation defined by $(z_n)_{n \in \mathbb{N}} \sim (z_n')_{n \in \mathbb{N}}$ if and only if $\lim_{n,m \to +\infty} (z_n,z_m')_{x} = + \infty$.  \linebreak
We will write $ z = [(z_n)] \in \partial X$ for short, and we say that $(z_n)$ {\em converges} to $z$. This definition  does not depend on the basepoint $x$. \\
There is a natural topology on $X\cup \partial X$ that extends the metric topology of $X$. 
The Gromov product can be extended to points $z,z'  \in \partial X$ by 
$$(z,z')_{x} = \sup_{(z_n) , (z_n') } \liminf_{n,m \to + \infty} (z_n, z_m')_{x}$$
where the supremum is taken among all sequences such that $(z_n) \approx z$ and $(z_n')\approx z'$.
For every $z,z',z'' \in \partial X$ it continues to hold
\begin{equation}
	\label{hyperbolicity-boundary}
	(z,z')_{x} \geq \min\lbrace (z,z'')_{x}, (z',z'')_{x} \rbrace - \delta.
\end{equation}
Moreover for all sequences $(z_n),(z_n')$ converging to  $z,z'$ respectively it holds
\begin{equation}
	\label{product-boundary-property}
	(z,z')_{x} -\delta \leq \liminf_{n,m \to + \infty} (z_n,z_m')_{x} \leq (z,z')_{x}.
\end{equation}
The Gromov product between a point $y\in X$ and a point $z\in \partial X$ is defined in a similar way and it  satisfies a condition analogue of  \eqref{product-boundary-property}.

\noindent Every geodesic ray $\xi$ defines a point  $\xi^+=[(\xi(n))_{n \in \mathbb{N}}]$  of the Gromov boundary $ \partial X$: we  say that $\xi$ {\em joins} $\xi(0) = y$ {\em to} $\xi^+ = z$, and we denote it by  $[y, z]$. Moreover for every $z\in \partial X$ and every $x\in X$ it is possible to find a geodesic ray $\xi$ such that $\xi(0)=x$ and $\xi^+ = z$. Indeed if $(z_n)$ is a sequence of points converging to $z$ then, by properness of $X$, the sequence of geodesics $[x,z_n]$ converges to a geodesic ray $\xi$ which has the properties above (cp. Lemma III.3.13 of \cite{BH09}). We denote any of these geodesic rays as $\xi_{xz} = [x,z]$ even if it is possibly not unique. \\
Given different points $z = [(z_n)], z' = [(z'_n)] \in \partial X$ there always exists  a geodesic line $\gamma$ joining $z$ to $z'$, i.e. such that  $\gamma|_{[0, +\infty)}$ and $\gamma|_{(-\infty,0]}$ join   $\gamma(0)$ to $z,z'$ respectively. We call $z$ and $z'$ the  {\em positive} and {\em negative endpoints} of $\gamma$, respectively,  denoted  $\gamma^+$ and $\gamma^-$. 
\vspace{2mm}

\noindent Here we recall some basic properties of Gromov-hyperbolic metric spaces.
\begin{lemma}[Projection Lemma, cp. Lemma 3.2.7 of \cite{CDP90}]
	\label{projection}
	Let $X$ be a $\delta$-hyperbolic metric space and let $x,y,z \in X$. For every geodesic segment $[y,z]$ we have $(y,z)_x \geq d(x, [y,z]) - 4\delta.$
\end{lemma}
\noindent We recall that a $(1,\nu)$-quasigeodesic is a curve $\alpha\colon I \to X$ such that 
$$\vert t - t' \vert - \nu \leq d(\alpha(t),\alpha(t')) \leq \vert t - t' \vert + \nu$$
for all $t,t'$ belonging to the interval $I$. As an immediate consequence of the previous lemma and Proposition 2.7 of \cite{CavS20bis} we get:
\begin{lemma}
	\label{projection-quasigeodesic}
	Let $X$ be a $\delta$-hyperbolic metric space, $x\in X$ and $\xi$ be a geodesic ray such that $\xi(0)$ is a projection of $x$ on $\xi$. Then
	\begin{itemize}
		\item[(i)] for all $T\geq 0$ any curve $\alpha=[x, \xi(0)] \cup \xi\vert_{[0,T]}$ is a $(1,4\delta)$-quasigeodesic. Moreover, if $\gamma$ is a geodesic segment $[x,\xi(T)]$, then $d(\alpha(t),\gamma(t))\leq 72\delta$ for all possibles $t$;
		\item[(ii)] any curve $\alpha=[x, \xi(0)] \cup \xi$ is a $(1,4\delta)$-quasigeodesic. Moreover, if $\xi'$ is a geodesic ray $[x,\xi^+]$, then $d(\alpha(t),\xi'(t))\leq 72\delta$ for all $t\geq 0$.
	\end{itemize}
\end{lemma}

Furthermore:
\begin{lemma}
	\label{approximation-line-ray}
	Let $X$ be a proper, $\delta$-hyperbolic metric space.
	Let $\gamma$ be a geodesic line and $x\in X$ with $S:=d(\gamma(0),x)$. Let $x'$ be a projection of $x$ on $\gamma$. Then
	\begin{itemize}
		\item[(i)] there exists an orientation of $\gamma$ such that $[x, x'] \cup [x', \gamma^+]$ is a $(1,4\delta)$-quasigeodesic, where the second segment is the subsegment of $\gamma$;
		\item[(ii)] with respect to the orientation of (i) then every geodesic ray $\xi=[x,\gamma^+]$ satisfies $d(\xi(S + t), \gamma(t)) \leq 76\delta$
		for all $t\geq 0$;
		\item[(iii)] for all orientations of $\gamma$ every geodesic ray $\xi = [x,\gamma^+]$ satisfies $d(\xi(S + t), \gamma(t)) \leq 2S + 76\delta$
		for all $t\geq 0$.
	\end{itemize}
\end{lemma}
\begin{proof}
	We choose the orientation of $\gamma$ such that $x'$ belongs to the negative ray $\gamma\vert_{(-\infty,0]}$ and we take a geodesic ray $\xi = [x, \gamma^+]$. By Lemma \ref{projection-quasigeodesic} the path $\alpha = [x,x'] \cup [x', \gamma^+]$ is a $(1,4\delta)$-quasigeodesic and moreover it satisfies $d(\xi(S + t), \alpha(S + t))\leq 72\delta$ for every $t\geq 0$.
	Furthermore the time $t_0$ such that $\alpha(t_0)=\gamma(0)$ is between $S$ and $S+4\delta$ implying $d(\xi(S + t), \gamma(t))\leq 76\delta.$\\
	For the second part of the proof we suppose to be in the situation above and we consider a geodesic ray $\xi=[x,\gamma^-]$. By Lemma \ref{projection-quasigeodesic} the path $\alpha = [x,x'] \cup [x', \gamma^-]$, where the second segment is a subsegment of $\gamma$, satisfies $d(\xi(S + t), \alpha(S + t))\leq 72\delta$ for every $t\geq 0$. Furthermore for every $t\geq 0$ the point $\alpha(S+t)$ belongs to $\gamma$ and $d(\alpha(S+t), \gamma(0)) \leq d(\alpha(S+t),x) + d(x, \gamma(0)) \leq 2S + t + 4\delta$.
	So $d(\xi(S+t), \gamma(t)) \leq 76\delta + 2S.$
\end{proof}
We remark that if $\gamma(0)$ is a projection of $x$ on $\gamma$ then the first part of the lemma holds for both the positive and negative ray of $\gamma$.\\
The {\em quasiconvex hull} of a subset $C$ of $\partial X$ is the union of all the geodesic lines joining two points of $C$ and it is denoted by QC-Hull$(C)$. The following is essentially Lemma 2.5 of \cite{Cav21ter}.

\begin{lemma}
	\label{approximation-ray-line}
	Let $X$ be a proper, $\delta$-hyperbolic metric space.
	Let $x\in X$ and $C\subseteq \partial X$ be a subset with at least two points. Then for every $z\in C$ it exists a geodesic line $\gamma$ with endpoints in $C$ such that $d(\xi_{xz}(t), \gamma(t))\leq 22\delta + d(x,\textup{QC-Hull}(C)) =: L$ for every $t\geq 0$.
\end{lemma}
\begin{proof}
	Let $x'\in \text{QC-Hull}(C)$ realizing $d(x, \text{QC-Hull}(C))$. By Lemma 2.5 of \cite{Cav21ter} we know there exists $\gamma$ as desired such that $d(\xi_{x'z}(t),\gamma(t))\leq 14\delta$ for every $t\geq 0$. Now the thesis follows by the fact that two geodesic rays with same endpoint are $8\delta$-close, see for instance Proposition 8.10 of \cite{BCGS}.
\end{proof}
\begin{obs}
	\label{C-subset-C'}
	Let $z,z'\in C\subseteq \partial X$. It is clear that the conclusion of Lemma \ref{approximation-ray-line} is true with $L = 22\delta + d(x,[z,z'])$.
\end{obs}

\begin{lemma}
	\label{bicombing-qchull}
	Let $(X,\sigma)$ be a proper, $\delta$-hyperbolic \textup{GCB}-space. Let \textup{QC-Hull}$_\sigma(C)$ be the union of all $\sigma$-geodesic lines joining two points of $C$. Then for every $x\in \textup{QC-Hull}(C)$ there exists $x' \in \textup{QC-Hull}_\sigma(C)$ with $d(x,x')\leq 8\delta$.
\end{lemma}
\begin{proof}
	For all $z,z'\in \partial X$ there exists a $\sigma$-geodesic line joining them, see \cite{CavS20bis}. Moreover two parallel geodesics are at most at distance $8\delta$ (\cite{BCGS}, Proposition 8.10), hence the conclusion.
\end{proof}

\subsection{Minkowski dimension}
\label{subsec-minkowski}
When $X$ is a proper, $\delta$-hyperbolic metric space we define the {\em generalized visual ball} of center $z \in \partial X$ and radius $\rho \geq 0$ to be
$$B(z,\rho) = \bigg\lbrace z' \in \partial X \text{ s.t. } (z,z')_{x} > \log \frac{1}{\rho} \bigg\rbrace.$$
It is comparable to the metric balls of the visual metrics on $\partial X$, see Lemma 2.6 of \cite{Cav21ter}.
Generalized visual balls are related to shadows. Let $x\in X$ be a basepoint. The shadow of radius $r>0$ casted by a point $y\in X$ with center $x$ is the set:
$$\text{Shad}_x(y,r) = \lbrace z\in \partial X \text{ s.t. } [x,z]\cap B(y,r) \neq \emptyset \text{ for all rays } [x,z]\rbrace.$$
\begin{lemma}[Lemma 2.7 of \cite{Cav21ter}]
	\label{shadow-ball}
	Let $X$ be a proper, $\delta$-hyperbolic metric space. Let $z\in \partial X$, $x\in X$ and $T\geq 0$. Then 
	\begin{itemize}
		\item[(i)] $B(z,e^{-T}) \subseteq \textup{Shad}_{x}\left(\xi_{xz}\left(T\right), 7\delta\right)$;
		\item[(ii)] $\textup{Shad}_{x}\left(\xi_{xz}\left(T\right), r\right) \subseteq B(z, e^{-T + r})$ for all $r> 0$.
	\end{itemize}
\end{lemma}
The {\em upper} and {\em lower visual Minkowski dimension} of a subset $C$ of $\partial X$ was defined in \cite{Cav21ter} as
$$\overline{\text{MD}}(C) = \limsup_{T  \to +\infty} \frac{1}{T}\log \text{Cov}(C, e^{-T}), \qquad \underline{\text{MD}}(C) = \liminf_{T  \to +\infty} \frac{1}{T}\log \text{Cov}(C, e^{-T})$$
respectively, where $\text{Cov}(C, \rho)$ denotes the minimal number of generalized visual balls of radius $\rho$ needed to cover $C$. Taking $C=\partial X$ we get
\begin{prop}
	\label{ent+mink}
	Let $(X,\sigma)$ be a $\delta$-hyperbolic \textup{GCB}-space that is $P_0$-packed at scale $r_0$ and let $x\in X$. Then 
	$$\frac{1}{T}\log \textup{Cov}(\partial X,e^{-T}) \underset{P_0,r_0,\delta}{\asymp} \frac{1}{T}\log \textup{Cov}(S(x, T),r_0).$$
	In particular the upper (resp. lower) visual Minkowski dimension of $\partial X$ equals the upper (resp. lower) covering entropy of $X$.
\end{prop}
We need:
\begin{lemma}[\cite{Cav21ter}, Lemma 2.2]
	\label{product-rays}
	Let $X$ be a proper, $\delta$-hyperbolic metric space, $z,z'\in \partial X$ and $x\in X$. Then
	\begin{itemize}
		\item[(i)] if $(z,z')_{x} \geq T$ then $d(\xi_{xz}(T - \delta),\xi_{xz'}(T - \delta)) \leq 4\delta$;
		\item[(ii)] for all $b> 0$, if $d(\xi_{xz}(T),\xi_{xz'}(T)) < 2b$ then $(z,z')_{x} > T - b$.
	\end{itemize}
\end{lemma}
\begin{proof}[Proof of Proposition \ref{ent+mink}]
	Let $z_1,\ldots, z_N$ be points of $\partial X$ realizing $\textup{Cov}(\partial X,e^{-T})$, and let $y_i$ be the point at distance $T$ from $x$ along one geodesic ray $\xi_{xz_i}$. We claim that $\lbrace y_i \rbrace$ covers $S(x,T)$ at scale $6\delta$. Indeed let $y\in S(x,T)$ and let $z \in \partial X$ be the point at infinity of a $\sigma$-geodesic ray $\xi$ that extends the $\sigma$-geodesic $[x,y]$. We know it exists $i$ such that $(z,z_i)_x > T$, then by Lemma \ref{product-rays} we get $d(y,y_i)\leq 6\delta$. This shows $\textup{Cov}(\partial X,e^{-T}) \geq \text{Cov}(S(x,T),6\delta)$.\\
	Now let $\lbrace y_i \rbrace$ be points realizing $\text{Cov}(S(x,T+\delta),\delta)$. For every $i$ let $z_i\in\partial X$ be the point at infinity of a $\sigma$-geodesic ray $\xi_i$ that extends the $\sigma$-geodesic $[x,y_i]$. For every $z\in \partial X$ we take a geodesic ray $\xi_{xz}$. We know it exists $i$ such that $d(\xi_{xz}(T+\delta),y_i)\leq \delta < 2\delta$, therefore $(z,z_i)_x > T$ by Lemma \ref{product-rays}. This shows $\textup{Cov}(\partial X,e^{-T}) \leq \text{Cov}(S(x,T+\delta),\delta)$. The conclusion follows by Proposition \ref{cov-ent-pack}.
\end{proof}

Putting together Proposition \ref{cov-ent-pack}, Proposition \ref{cov-vol-asym}, Proposition \ref{ent+mink}, Theorem \ref{lip-top_cov} and Proposition \ref{entropyrelations} we get the proof of Theorem \ref{B}.

\section{Entropies of subsets of the boundary}
\label{sec-closed-subsets}
Let $(X,\sigma)$ be a $\delta$-hyperbolic \textup{GCB}-space that is $P_0$-packed at scale $r_0$. In this section we will consider a subset $C$ of $\partial X$ and we define the relative version, with respect to $C$, of all the different definitions of entropies introduced in the previous sections. We observe that when $C = \partial X$ then we are in the case yet studied. 

\subsection{Covering and volume entropy}
\label{subsec-closed-cov-vol}
Let $(X,\sigma)$ be a $\delta$-hyperbolic \textup{GCB}-space that is $P_0$-packed at scale $r_0$ and let $C$ be a subset of $\partial X$.
The {\em upper covering entropy} of $C$ is defined as
$$\limsup_{T\to +\infty} \frac{1}{T} \log \text{Cov}(\overline{B}(x,T) \cap \overline{B}(\text{QC-Hull}(C),\tau), r),$$
where $r > 0$, $\tau \geq 0$ and $x\in X$ and it is denoted by $\overline{h_{\text{Cov}}}(C)$.
The {\em lower covering entropy} of $C$, denoted by $\underline{h_{\text{Cov}}}(C)$, is defined taking the limit inferior instead of the limit superior. These quantities do not depend on $x \in X$ as usual.
The analogous of Proposition \ref{cov-ent-pack} holds.
\begin{prop}
	\label{cov-ent-closed}
	Let $(X,\sigma)$ be a $\delta$-hyperbolic \textup{GCB}-space that is $P_0$-packed at scale $r_0$, $C$ be a subset of $\partial X$ and $x\in X$. Then
	\begin{equation*}
		\begin{aligned}
			&\frac{1}{T}\log \textup{Cov}(\overline{B}(x,T) \cap \overline{B}(\textup{QC-Hull}(C),\tau), r) \underset{P_0,r_0,r,r', \tau, \tau'}{\asymp}\\
			&\frac{1}{T}\log \textup{Pack}(\overline{B}(x,T) \cap \overline{B}(\textup{QC-Hull}(C),\tau'), r')
		\end{aligned}
	\end{equation*}
	for all $r, r' > 0$ and $\tau,\tau'\geq 0$.
	In particular any of these functions can be used in the definition of the upper and lower covering entropies of $C$.
\end{prop}
\begin{proof}
	Once $\tau$ is fixed the asymptotic estimate can be proved exactly as in Proposition \ref{cov-ent-pack}. Moreover for all $\tau \geq 0$ it is easy to prove that
	\begin{equation*}
		\begin{aligned}
			&\text{Cov}(\overline{B}(x,T) \cap \overline{B}(\text{QC-Hull}(C),\tau), r) \leq\\
			&\text{Cov}(\overline{B}(x,T)\cap \text{QC-Hull}(C), r_0) \cdot \text{Cov}(r_0+\tau,r_0).
		\end{aligned}
	\end{equation*}
	and $\text{Cov}(r_0+\tau,r_0)$ is uniformly bounded in terms of $P_0,r_0$ and $\tau$ by Proposition \ref{packingsmallscales}. This concludes the proof.
\end{proof}
\noindent Clearly when $C=\partial X$ we have $\overline{h_{\textup{Cov}}}(\partial X) = \overline{h_{\textup{Cov}}}(X).$
Moreover if $C$ is a closed subset of $\partial X$ then $\overline{h_{\textup{Cov}}}(C) \leq \overline{h_{\textup{Cov}}}(\partial X),$ so $\overline{h_{\textup{Cov}}}(C) \leq \frac{\log (1+P_0)}{r_0}$ by Lemma \ref{entropy-bound}.\\
The analogous of Proposition \ref{spherical-entropy} holds. We remark that in this case a dependence on $\delta$ appears.
\begin{prop}
	\label{closed-spherical-entropy}
	Let $(X,\sigma)$ be a $\delta$-hyperbolic \textup{GCB}-space that is $P_0$-packed at scale $r_0$, $C$ be a subset of $\partial X$ and $x\in X$. Then
	$$\frac{1}{T}\log \textup{Cov}(\overline{B}(x,T)\cap \textup{QC-Hull}(C), r)\underset{P_0,r_0,r, \delta}{\asymp}\frac{1}{T}\log \textup{Cov}(S(x,T)\cap\textup{QC-Hull}(C),r)$$
	In particular any of these functions can be used in the definition of the upper and lower covering entropies of $C$.
\end{prop}
\begin{proof}
	As in the proof of Proposition \ref{spherical-entropy} one inequality is obvious, so we are going to prove the other.
	We divide the ball $\overline{B}(x,T)$ in the annulii $A(x,kr, (k+1)r)$ with $k=0, \ldots, \frac{T}{r}-1$.
	Therefore we can estimate the quantity $\text{Cov}(\overline{B}(x,T)\cap \text{QC-Hull}(C), 72\delta + 2r)$ from above by
	$$\sum_{k=0}^{\frac{T}{r}-1}\text{Cov}(A(x,kr,(k+1)r) \cap \text{QC-Hull}(C), 72\delta + 2r).$$
	We claim that every element of the sum is 
	$\leq \text{Cov}(S(x,T)\cap\text{QC-Hull}(C),r).$
	Indeed let $y_1,\ldots, y_N$ be a set of points realizing $\text{Cov}(S(x,T) \cap \text{QC-Hull}(C),r)$. For all $i=1,\ldots,N$ we consider the $\sigma$-geodesic segment $\gamma_i=[x,y_i]$ and we call $x_i$ the point along this geodesic at distance $kr$ from $x$. We want to show that $x_1,\dots,x_N$ is a $(72\delta + 2r)$-dense subset of $A(x,kr,(k+1)r) \cap \text{QC-Hull}(C)$. Given a point $y \in A(x,kr,(k+1)r)\cap \text{QC-Hull}(C)$
	there exists a $\sigma$-geodesic line $\gamma$ with endpoints in $C$ containing $y$. We parametrize $\gamma$ so that $\gamma(0)$ is a projection of $x$ on $\gamma$ and $y\in \gamma\vert_{[0,+\infty)}$. We take a point $y_T \in \gamma\vert_{[0,+\infty)}$ at distance $T$ from $x$, so that $y_T\in S(x,T)\cap\text{QC-Hull}(C)$ and therefore there exists $i$ such that $d(y_T,y_i)\leq r$.
	By Lemma \ref{projection-quasigeodesic} the path $\alpha = [x,\gamma(0)] \cup [\gamma(0), y_T]$, where the second geodesic is a subsegment of $\gamma$, is a $(1, 4\delta)$-quasigeodesic and, if $t_y$ denotes the real number such that $\alpha(t_y)=y$, it holds $t_y \in [kr, (k+1)r].$	
	By Lemma \ref{projection-quasigeodesic} we get $d(y,\gamma_i'(t_y))\leq 72\delta$, where $\gamma_i'$ is the $\sigma$-geodesic $[x,y_T]$.
	We conclude the proof of the claim since 
	$$d(y,x_i)\leq d(y,\gamma_i'(t_y)) + d(\gamma_i'(t_y), \gamma_i(t_y))+ d(\gamma_i(t_y), x_i) \leq 72\delta + 2r,$$ from the convexity of $\sigma$. 
	The thesis follows by Proposition \ref{cov-ent-closed}.
\end{proof}

The {\em upper volume entropy} of $C$ with respect to a measure $\mu$ is
$$\overline{h_\mu}(C) = \sup_{\tau \geq 0}\limsup_{T\to + \infty}\frac{1}{T}\log \mu(\overline{B}(x,T) \cap\overline{B}(\text{QC-Hull}(C),\tau)),$$
where $x \in X$.
Taking the limit inferior instead of the limit superior is defined the {\em lower volume entropy} of $C$, $\underline{h_\mu}(C)$.

\begin{prop}
	\label{closed-volume-entropy}
	Let $(X,\sigma)$ be a $\delta$-hyperbolic \textup{GCB}-space that is $P_0$-packed at scale $r_0$, let $C$ be a subset of $\partial X$ and let $\mu$ be a measure on $X$ which is $H$-homogeneous at scale $r$. Then for all $\tau \geq r$ it holds
	\begin{equation*}
		\begin{aligned}
			&\frac{1}{T}\log\mu(\overline{B}(x,T) \cap\overline{B}(\textup{QC-Hull}(C),\tau))\underset{H,P_0,r_0,r,\tau}{\asymp}\\
			&\frac{1}{T}\log \textup{Cov}(\overline{B}(x,T) \cap \textup{QC-Hull}(C), r_0).
		\end{aligned}
	\end{equation*}
	In particular the upper (resp. lower) volume entropy of $C$ with respect to $\mu$ coincides with the upper (resp. lower) covering entropy of $C$ and it can be computed using $\tau = r$ in place of the supremum.
\end{prop}
\begin{proof}
	By Remark \ref{homogeneous} we know that $\mu$ is $H(\tau)$-homogeneous at scale $\tau$ for all $\tau\geq r$, where $H(\tau)$ depends on $P_0,r_0,\tau,r,H$. Therefore the proof of Proposition \ref{cov-vol-asym} works in this case.
\end{proof}

\subsection{Lipschitz topological entropy}
\label{subsec-closed-liptop}
Let $(X,\sigma)$ be a $\delta$-hyperbolic \textup{GCB}-space that is $P_0$-packed at scale $r_0$.
For a subset $C$ of $\partial X$ and $Y\subseteq X$ we set
$$\text{Geod}_\sigma(Y,C)= \lbrace \gamma \in \text{Geod}_\sigma(X) \text{ s.t. } \gamma^{\pm} \subseteq C \text{ and } \gamma(0)\in Y\rbrace.$$ 
If $Y = X$ we simply write $\text{Geod}_\sigma(C)$.
Clearly $\text{Geod}_\sigma(C)$ is a $\Phi$-invariant subset of $\text{Geod}_\sigma(X)$, so the reparametrization flow is well defined on it. 
The {\em upper Lipschitz-topological entropy} of $\text{Geod}_\sigma(C)$ is defined as
$$\overline{h_{\text{Lip-top}}}(\text{Geod}_\sigma(C)) = \inf_{\text{\texthtd}}\sup_K \lim_{r\to 0} \limsup_{T\to + \infty}\frac{1}{T}\log\text{Cov}_{\text{\texthtd}^T}(K,r),$$
where the infimum is taken among all geometric metrics on $\text{Geod}_\sigma(C)$. The {\em lower Lipschitz-topological entropy} is defined taking the limit inferior instead of the limit superior and it is denoted by $\underline{h_{\text{Lip-top}}}(\text{Geod}_\sigma(C))$.
In the following result we observe the difference between closed and non-closed subsets of $\partial X$.
\begin{theo}
	\label{lipschitz-entropy}
	Let $(X,\sigma)$ be a $\delta$-hyperbolic \textup{GCB}-space that is $P_0$-packed at scale $r_0$ and $C$ be a subset of $\partial X$. Then
	$$\overline{h_{\textup{Lip-top}}}(\textup{Geod}_\sigma(C)) = \sup_{C'\subseteq C}\overline{h_{\textup{Cov}}}(C'),$$
	where the supremum is among closed subsets $C'$ of $C$.
	The same holds for the lower entropies.
\end{theo}
\noindent We remark that the supremum of the covering entropies among the closed subsets of $C$ can be strictly smaller than the covering entropy of $C$ (see \cite{Cav21bis}), marking the distance between the equivalences of the different notions of entropies in case of non-closed subsets of the boundary.
We start with an easy lemma.
\begin{lemma}
	Let $(X,\sigma)$ and $C$ be as in Theorem \ref{lipschitz-entropy} and let $x\in X$. Then every compact subset of $\textup{Geod}_\sigma(C)$ is contained in $\textup{Geod}_\sigma(\overline{B}(x,R), C')$ for some $R\geq 0$ and some $C'\subseteq C$ closed. Moreover $\textup{Geod}_\sigma(\overline{B}(x_0,R), C')$ is compact for all $R\geq 0$ and all closed $C'\subseteq C$.
\end{lemma}
\begin{proof}
	We fix a compact subset $K$ of $\textup{Geod}_\sigma(C)$. The continuity of the evaluation map gives that $E(K)$ is contained in some ball $\overline{B}(x,R)$. Moreover the maps $+,-\colon \text{Geod}_\sigma(X)\to \partial X$, defined by $\gamma \mapsto \gamma^+, \gamma^-$ respectively, are continuous (\cite{BL12}, Lemma 1.6). This means that $C' = +(K)\cup -(K)$ is a closed subset of $\partial X$ and clearly $K\subseteq \text{Geod}_\sigma(\overline{B}(x,R), C')$.
	By a similar argument, and since the evaluation map is proper, it follows that the set $\text{Geod}_\sigma(\overline{B}(x,R), C')$ is compact for all $R\geq 0$ and all $C'\subseteq C$ closed.
\end{proof}
For a metric $f\in \mathcal{F}$ and $C\subseteq \partial X$ we denote by $\overline{h_f}$ the upper metric entropy of Geod$_\sigma(C)$ with respect to $f$, that is
$$\overline{h_f}(\text{Geod}_\sigma(C))=\sup_K \lim_{r\to 0} \limsup_{T\to + \infty}\frac{1}{T}\log\text{Cov}_{f^T}(K,r).$$
Taking the limit inferior instead of the limit superior we define the lower metric entropy of Geod$_\sigma(C)$ with respect to $f$, denoted by  $\underline{h_f}(\text{Geod}_\sigma(C))$.
The analogous of Proposition \ref{entropyrelations} is the following.
\begin{prop}
	\label{closed-entropyrelations}
	Let $(X,\sigma)$ be as in Theorem \ref{lipschitz-entropy}, $C'$ be a closed subset of $\partial X$, $f\in \mathcal{F}$, $x\in X$ and $L$ be the constant given by Lemma \ref{approximation-ray-line}. Then
	\begin{itemize}
		\item[(i)] $\overline{h_f}(\textup{Geod}_\sigma(\overline{B}(x,R), C')) = \overline{h_f}(\textup{Geod}_\sigma(\overline{B}(x,L), C'))$ for all $R\geq L$; 
		\item[(ii)] $\overline{h_f}(\textup{Geod}_\sigma(C')) = \overline{h_f}(\textup{Geod}_\sigma(\overline{B}(x,L), C')) \leq \overline{h_{\textup{Cov}}}(C');$
		\item[(iii)] The function $r\mapsto \limsup_{T\to +\infty}\frac{1}{T}\log \textup{Cov}_{f^T}(\textup{Geod}_\sigma(\overline{B}(x,L), C'), r)$
		is constant.
	\end{itemize}
	The same conclusions hold for the lower entropies.
\end{prop}
We observe that applying the Key Lemma \ref{residualentropy} we have directly the relative version of Corollary \ref{entropyrescale}.
\begin{cor}
	\label{closed-entropyrescale}
	Let $f\in \mathcal{F}$, $x\in X$, $R\geq 0$ and $0 < r \leq r'$. Then
	\begin{equation*}
		\begin{aligned}
			&\frac{1}{T}\log \textup{Cov}_{f^T}(\textup{Geod}_\sigma(\overline{B}(x,R), C'),r') \underset{P_0,r_0,r,r',f}{\asymp}\\ &\frac{1}{T} \log \textup{Cov}_{f^T}(\textup{Geod}_\sigma(\overline{B}(x,R), C'),r).
		\end{aligned}
	\end{equation*}
	
\end{cor}

\begin{proof}[Proof of of Proposition \ref{closed-entropyrelations}.]
	We fix $R\geq L$ and $T\geq 0$. We take a set $\gamma_1, \ldots, \gamma_N$ of $\sigma$-geodesic lines realizing
	$\text{Cov}_{f^T}(\text{Geod}_\sigma(\overline{B}(x,L),C'), r_0).$
	Our aim is to show that $\gamma_1,\ldots, \gamma_N$ is a $(4R + 2L + C(f) + 76\delta + r_0)$-dense subset of $\text{Geod}_\sigma(\overline{B}(x,R),C').$ This, together with Corollary \ref{closed-entropyrescale}, will prove (i).\linebreak
	We consider a $\sigma$-geodesic line $\gamma \in \text{Geod}_\sigma(\overline{B}(x,R),C')$, so $d(\gamma(0), x) =: S \leq R$. By Lemma \ref{approximation-line-ray} there exists a $\sigma$-geodesic ray $\xi$ starting at $x$ such that 
	$d(\xi(S + t), \gamma(t))\leq 2S + 76\delta$ 
	for all $t\geq 0$ and in particular $\xi^+$ belongs to $C$. Now we apply Lemma \ref{approximation-ray-line} to find a $\sigma$-geodesic line $\gamma' \in \text{Geod}(C')$ such that
	$d(\xi(t), \gamma'(t))\leq L$
	for all $t\geq 0$.
	Clearly we have $\gamma' \in \text{Geod}_\sigma(\overline{B}(x,L), C')$ and
	$d(\gamma'(S + t), \gamma(t)) \leq 2S + L + 76\delta$ for all $t\geq 0$. Therefore
	$d(\gamma'(t), \gamma(t)) \leq 3S + L + 76\delta$ for all $t\geq 0$.
	This implies that for all $t\in [0,T]$ we have
	$$f^t(\gamma, \gamma')
	\leq\int_{-\infty}^{-t}\big( d(\gamma(0),\gamma'(0)) + 2\vert s \vert\big)f(s)ds + \int_{-t}^{+\infty}\big(3S + L + 76\delta\big)f(s)ds.$$
	Since $d(\gamma(0),\gamma'(0)) \leq L + S$ we get
	$f^t(\gamma, \gamma') \leq 4S + 2L + C(f) + 76\delta$ using the properties of $f$, and so $f^T(\gamma, \gamma')\leq 4R + 2L + C(f) + 76\delta.$
	Moreover, since $\gamma' \in \text{Geod}_\sigma(\overline{B}(x,L),C')$, there exists $\gamma_i$ such that $f^T(\gamma',\gamma_i)\leq r_0$. This implies 
	$f^T(\gamma, \gamma_i)\leq 4R + 2L + C(f) + 76\delta + r_0.$\\
	We observe that (iii) follows directly from the previous corollary. \\
	The first equality in (ii) follows by (i). 
	In order to prove the inequality we fix $y_1,\ldots,y_N$ realizing Cov$(S(x,T)\cap \text{QC-Hull}(C'), r_0)$. Up to change $y_i$ with a point at distance at most $8\delta$ from it we can suppose there are $\gamma_i \in \text{Geod}_\sigma(C')$ such that $y_i \in \gamma_i$ and $y_1,\ldots,y_N$ is a $(8\delta + r_0)$-dense subset of $S(x,T)\cap \text{QC-Hull}(C')$, as follows by Lemma \ref{bicombing-qchull}. By Lemma \ref{approximation-line-ray} there exists an orientation of $\gamma_i$ such that, called $S_i=d(x,\gamma_i(0))$ and $T_i \geq 0$ such that $\gamma_i(T_i) = y_i$, we have $T\leq S_i + T_i \leq T + 4\delta$ and the $\sigma$-geodesic ray $\xi_i = [x,\gamma_i^+]$ satisfies $d(\xi_i(S_i+t), \gamma_i(t))\leq 76\delta$ for all $t\geq 0$. By Lemma \ref{approximation-ray-line} there exists $\gamma_i'\in \text{Geod}_\sigma(\overline{B}(x,L), C')$ such that $d(\gamma_i'(t), \xi_i(t))\leq L$ for all $t\geq 0$. We claim that the set $\lbrace\gamma_i'\rbrace$ is $(6L + 176\delta + 2r_0 + 2C(f))$-dense in $\text{Geod}_\sigma(\overline{B}(x,L), C')$. By (i) and (iii) this would imply the thesis.
	We fix $\gamma \in \text{Geod}_\sigma(\overline{B}(x,L), C')$, so there exists $y \in S(x,T)$ and $T_y\in [T-L,T+L]$ such that $\gamma(T_y)=y$ and therefore $d(y,y_i)\leq 8\delta + r_0$ for some $i$.
	We observe that we have $d(\gamma_i'(S_i + T_i), y_i) \leq L + 76\delta$ and so $d(\gamma_i'(T), y_i) \leq L + 80\delta$. Moreover $d(\gamma(T),y_i)\leq L + 8 \delta + r_0$ implying $d(\gamma(T), \gamma_i'(T)) \leq 2L + 88\delta + r_0$. Furthermore by definition $d(\gamma(0),\gamma_i'(0))\leq 2L$, so by convexity of $\sigma$ we get $d(\gamma(t), \gamma_i'(t))\leq 2L + 88\delta + r_0$ for all $t\in [0,T]$. The thesis follows by the classical subdivision of the integral defining $f$ into three parts, each estimated by the constants above.
\end{proof}

\begin{proof}[Proof of Theorem \ref{lipschitz-entropy}]
	We fix a geometric metric \texthtd\, on $\text{Geod}_\sigma(C)$ and we denote by $M$ the Lipschitz constant with respect to \texthtd\, of the evaluation map $E$. By Remark \ref{C-subset-C'} the constant $L$ given by Lemma \ref{approximation-ray-line} can be chosen independently of $C' \subseteq C$, once $x$ is fixed. 	
	Clearly we have
	$$\sup_{R\geq 0, C'\subseteq C} \lim_{r\to 0} \limsup_{T\to + \infty}\frac{1}{T}\log\text{Cov}_{\text{\texthtd}^T}(\text{Geod}_\sigma(\overline{B}(x,R),C'),r) \geq$$
	$$\sup_{C'\subseteq C}\limsup_{T\to + \infty}\frac{1}{T}\log\text{Cov}_{\text{\texthtd}^T}(\text{Geod}_\sigma(\overline{B}(x,L), C'),r_0).$$
	We fix $\sigma$-geodesic lines $\gamma_1,\ldots,\gamma_N$ realizing Cov$_{\text{\texthtd}^T}(\text{Geod}_\sigma(\overline{B}(x,L), C'), r_0)$. Since $d(\gamma_i(0), x)\leq L$ for all $i=1,\ldots,N$ then there exists $t_i \in [T-L,T+L]$ such that $d(\gamma_i(t_i), x)=T$. We claim that the points $y_i =\gamma_i(t_i) \in S(x,T)\cap \text{QC-Hull}(C')$ are $(2L + 80\delta + Mr_0)$-dense. By Proposition \ref{closed-spherical-entropy} this would imply
	$$\overline{h_{\text{Lip-top}}}(\text{Geod}_\sigma(C)) \geq \sup_{C'\subseteq C} \overline{h_{\text{Lip-top}}}(\text{Geod}_\sigma(C')) \geq \sup_{C'\subseteq C} \overline{h_{\text{Cov}}}(C').$$
	We fix $y\in S(x,T)\cap \text{QC-Hull}(C')$ and we select a geodesic line $\gamma\in \text{Geod}(C')$ containing $y$. Up to change $y$ with a point at distance at most $8\delta$ we can suppose $\gamma\in \text{Geod}_\sigma(C')$, as follows by Lemma \ref{bicombing-qchull}. By Lemma \ref{approximation-line-ray}, with an appropriate choice of the orientation of $\gamma$, the $\sigma$-geodesic ray $\xi=[x,\gamma^+]$ satisfies $d(\xi(S+t), \gamma(t))\leq 76\delta$ for all $t\geq 0$, where $S=d(x,\gamma(0))$. By Lemma \ref{approximation-ray-line} there exists $\gamma'\in \text{Geod}_\sigma(\overline{B}(x,L), C')$ such that $d(\xi(t), \gamma'(t))\leq L$ for all $t\geq 0$, implying $d(\gamma'(S+t),\gamma(t))\leq L + 76\delta$ for all $t\geq 0$. Denoting by $T_y$ the real number such that $\gamma(T_y) = y$ we have by Lemma \ref{approximation-line-ray} that $T\leq S + T_y \leq T + 4\delta$. Therefore we apply the previous estimate with $t=T_y$ obtaining $d(\gamma'(T), y) \leq d(\gamma'(T),\gamma'(S + T_y)) + d(\gamma'(S+T_y), y) \leq L + 80\delta$.
	Moreover there exists $i\in \lbrace 1,\ldots,N\rbrace$ such that $\text{\texthtd}^T(\gamma',\gamma_i)\leq r_0$ and in particular $d(\gamma'(T),\gamma_i(T))\leq Mr_0$. Therefore we get $d(y_i, y)\leq d(\gamma_i(t_i), \gamma_i(T)) + d(\gamma_i(T), y) \leq 2L + 80\delta + Mr_0$. Now, up to add $8\delta$, we  obtain the inequality.
The other inequality follows by Proposition \ref{closed-entropyrelations}. Indeed we have
$$\overline{h_{\text{Lip-top}}}(\text{Geod}_\sigma(C)) \leq \sup_{C'\subseteq C} \overline{h_f}(\text{Geod}_\sigma(C')) \leq \sup_{C'\subseteq C} \overline{h_{\text{Cov}}}(C').$$
\end{proof}
\begin{obs}
	\label{closed-asymp-geodesic}
	Let $(X,\sigma)$ be as in Theorem \ref{lipschitz-entropy}, $C\subseteq \partial X$ closed and \linebreak 
	$x\in \textup{QC-Hull}(C)$. By the proof of Theorem \ref{lipschitz-entropy}, Lemma \ref{approximation-ray-line} and Remark \ref{C-subset-C'} we obtain
	$$\frac{1}{T}\log \textup{Cov}(S(x,T)\cap \textup{QC-Hull}(C), r_0) \underset{P_0, r_0, \delta, f}{\asymp}\frac{1}{T}\log \textup{Cov}_{f^T}(\textup{Geod}_\sigma(\overline{B}(x,L), C), r_0)$$
	for all $f\in \mathcal{F}$, where $L=14\delta$.
\end{obs}

\subsection{Minkowski dimension}
\label{subsec-closed-minkowski}
The relative version of Proposition \ref{ent+mink} is:
\begin{prop}
	\label{closed-ent-mink}
	Let $(X,\sigma)$ be a $\delta$-hyperbolic \textup{GCB}-space that is $P_0$-packed at scale $r_0$, let $C$ be a subset of $\partial X$, $x\in X$ and $L$ be the constant given by Lemma \ref{approximation-ray-line}. Then 
	$$\frac{1}{T}\log \textup{Cov}(C,e^{-T}) \underset{P_0,r_0,\delta,L}{\asymp} \frac{1}{T}\log \textup{Cov}(S(x, T)\cap \textup{QC-Hull}(C),r_0).$$
	In particular the upper (resp. lower) Minkowski dimension of $C$ equals the upper (resp. lower) covering entropy of $C$.
\end{prop}
\begin{proof}
	We can suppose $T\geq L$. Let $z_1,\ldots,z_N$ be points realizing $\textup{Cov}(C,e^{-T})$. For every $i$ we take a geodesic ray $[x,z_i]$. By Lemma \ref{approximation-ray-line} there exists a geodesic line $\gamma_i$ with both endpoints in $C$ such that $d(\gamma_i(t),\xi(t)) \leq L$ for every $t\geq 0$. We take a point $y_i = \gamma_i(t_i)$ with $t_i\geq 0$ such that $d(y_i,x)=T$. We know $\vert t_i - T \vert \leq 2L$. We claim the set $\lbrace y_i \rbrace$ cover $S(x,T)\cap \text{QC-Hull}(C)$ at scale $82\delta + 3L$.
	Indeed let $y\in S(x,T)\cap \text{QC-Hull}(C)$, i.e. it exists a geodesic line $\gamma$ with both endpoints in $C$ such that $y\in \gamma$. We parametrize $\gamma$ so that $\gamma(0)$ is a projection of $x$ on $\gamma$ and $y\in \gamma\vert_{[0,+\infty)}$. We consider a geodesic ray $\xi = [x,\gamma^+]$. By Lemma \ref{approximation-line-ray} we know that $d(\xi(T),y)\leq 76\delta + 3L$. Moreover $\gamma^+\in C$, so there is $i$ such that $(z_i,\gamma^+)> T$. By Lemma \ref{product-rays} we get $d(y_i, \xi(T))\leq 6\delta$, so $d(y,y_i)\leq 82\delta + 3L$. This shows
	$$\text{Cov}(S(x,T)\cap \text{QC-Hull}(C), 82\delta + 3L) \leq \text{Cov}(C,e^{-T}).$$
	Let $\lbrace y_1,\ldots, y_N\rbrace$ be points realizing $\text{Cov}(S(x,T+2L + 38\delta)\cap \text{QC-Hull}(C), \frac{L}{2})$. Therefore for every $i$ it exists a geodesic line $\gamma_i$ with both endpoints in $C$ containing $y_i$. We parametrize each $\gamma_i$ so that $\gamma_i(0)$ is a projection of $x$ on $\gamma_i$ and $y_i\in \gamma_i\vert_{[0,+\infty)}$. We claim that the set $\lbrace\gamma_i^+\rbrace$ cover $C$ at scale $e^{-T}$. Indeed for every $z\in C$ we take a geodesic ray $\xi = [x,z]$ and we set $y = \xi(T + 2L + 38\delta)$. By Lemma \ref{approximation-ray-line} we know it exists a geodesic line $\gamma$ with both endpoints in $C$ such that $d(y,\gamma(T + 2L + 38\delta))\leq L$. Moreover there is a point $y'$ along $\gamma$ which is at distance exactly $T + 2L + 38\delta$ from $x$ and that satisfies $d(y,y')\leq 2L$. Now we know it exists $i$ such that $d(y', y_i)\leq \frac{L}{2}$, moreover for a fixed geodesic ray $\xi_i = [x,z_i]$ it holds $d(\xi_i(T + 2L + 38\delta), y_i) \leq L + 76\delta$ by Lemma \ref{approximation-line-ray}. So in conclusion we get 
	$$d(\xi(T + 2L + 38\delta), \xi_i(T + 2L + 38\delta)) < 4L + 76\delta.$$ 
	By Lemma \ref{product-rays} we conclude that $(z,z_i)_x > T$, i.e. 
	$$\text{Cov}(C, e^{-T}) \leq \text{Cov}\left(S(x,T+2L + 38\delta)\cap \text{QC-Hull}(C), \frac{L}{2}\right).$$
	Now the conclusion follows by Proposition \ref{cov-ent-closed}.
\end{proof}

The proof of Theorems \ref{E} and \ref{G} follow by Proposition \ref{cov-ent-closed}, Proposition \ref{closed-spherical-entropy}, Proposition \ref{closed-volume-entropy}, Theorem \ref{lipschitz-entropy}, Proposition \ref{closed-entropyrelations}, Remark \ref{closed-asymp-geodesic} and Proposition \ref{closed-ent-mink}.

		\clearpage
		\bibliographystyle{alpha}
		\bibliography{Entropies_of_non_positively_curved_metric_spaces}
		
	\end{document}